\documentclass[12pt]{article}
\usepackage{amsmath,amssymb,bm,fixmath,slashed,amsfonts,amsthm}
\usepackage{amsthm}
\usepackage[latin2]{inputenc}                                          
\usepackage{graphicx}
\theoremstyle{plain}

\newtheorem{theorem}{Theorem}
\newtheorem*{theorem*}{Theorem}

\newtheorem{corollary}[theorem]{Corollary}

\newtheorem{lemma}[theorem]{Lemma}

\newtheorem{proposition}[theorem]{Proposition}

\theoremstyle{definition}
\newtheorem{definition}[theorem]{Definition}
\theoremstyle{remark}
\newtheorem{remark}[theorem]{Remark}

\newcommand{\de}{\mathbb \delta}

\newcommand{\La}{\mathbb \Lambda}

\newcommand{\om}{\omega}

\newcommand{\mP}{\mathbb P}

\newcommand{\mC}{\mathbb C}

\def\ga{\gamma}
\def\de{\delta}

\def\rh{\rho}
\def\si{\sigma}

\def\ph{\varphi}
\def\ps{\psi}
\def\om{\omega}
\def\Ga{\Gamma}

\def\La{\Lambda}
\def\Si{\Sigma}
\def\Ph{\Phi}
\def\Up{\Upsilon}

\newcommand{\lA}{{\cal A}}

\newcommand{\mR}{\mathbb R}
 \newcommand{\mN}{\mathbb N}
\newcommand{\be}{\begin{eqnarray}}
\newcommand{\ee}{\end{eqnarray}}
\newcommand{\bd}{\begin{definition}}
\newcommand{\ed}{\end{definition}}
\newcommand{\br}{\begin{remark}}
\newcommand{\er}{\end{remark}}

\newcommand{\gog}{{\mathfrak g}}

\newcommand{\bl}{\begin{lemma}}
\newcommand{\el}{\end{lemma}}

\newcommand{\bp}{\begin{picture}}
\newcommand{\ep}{\end{picture}}
\newcommand{\bi}{\begin{itemize}}
\newcommand{\ei}{\end{itemize}}
\newcommand{\bq}{\begin{quotation}}
\newcommand{\eq}{\end{quotation}}

\newcommand{\End}{\operatorname{End}}

\newcommand{\Diff}{\operatorname{Diff}}

\newcommand{\cA}{\mathcal{A}}
\newcommand{\cE}{\mathcal{E}}
\newcommand{\cO}{\mathcal{O}}
\newcommand{\cS}{\mathcal{S}}
\newcommand{\cQ}{\mathcal{Q}}
\newcommand{\na}{\nabla}
\newcommand{\Rho}{{\mbox{\sf P}}}
\newcommand{\Ric}{{\mbox{\sf Ric}}}
\newcommand{\bfom}{\pmb{\omega}}
\newcommand{\bfga}{\pmb{\gamma}}
\newcommand{\bfep}{\pmb{\epsilon}}
\newcommand{\Dir}{\slashed{D}}
\newcommand{\id}{\mathrm{id}}
\newcommand{\tf}{\mathrm{tf}}
\newcommand{\ol}[1]{\overline{#1}}
\newcommand{\wh}[1]{\widehat{#1}}

\newcommand{\lpl}                         
{\mbox{$
\begin{picture}(12.7,8)(-.5,-1)
\put(2,0.2){$+$}
\put(6.2,2.8){\oval(8,8)[l]}
\end{picture}$}}

\begin{document}

\title{Higher symmetries of symplectic Dirac operator}
\author{ Petr Somberg and Josef \v{S}ilhan}
\date{}
\maketitle
\date{}
\abstract 

We construct in projective 
differential geometry of the real dimension $2$
higher symmetry algebra of the symplectic Dirac operator 
$\Dir_s$ acting on symplectic spinors.
The higher symmetry differential operators correspond 
to the solution space 
of a class of projectively invariant
overdetermined operators of arbitrarily high order acting on 
symmetric tensors. The higher symmetry algebra structure corresponds
to a completely prime primitive ideal having as its associated 
variety the minimal nilpotent orbit of $\mathfrak{sl}(3,\mR)$.

\vspace{0.5cm}
	
  

\endabstract

\tableofcontents  

\section{Introduction}
 \numberwithin{equation}{section}
 \setcounter{equation}{0}

\hspace{0.4cm} It is always desirable to convert a purely 
algebraic construct into its geometrical
realization with the hope to gain its better 
understanding, as well as a potential 
generalization of the former algebraic 
structure. The present article is an example of this phenomenon: the algebraic 
structure is the algebra of higher symmetries of the symplectic Dirac operator $\Dir_s$
(cf. the seminal work \cite{kos}) realized in projective differential geometry 
of the real dimension two. This algebra corresponds to a completely prime primitive ideal
which has as its associated variety the minimal nilpotent orbit of the complexification 
of $\mathfrak{sl}(3,\mR)$, while the geometric realization pursued in our article 
relies on the use 
of certain class of projectively invariant systems of differential equations of
arbitrarily high order together
with the convenient calculus of symmetric powers of projective adjoint tractor bundle.     
Consequently, our geometric construction can be generalized to any smooth 
manifold equipped with a projective differential structure, thereby making contact with the 
appearance of linear and bilinear differential invariants for a given (curved) projective
structure.

There is a well established notion of higher symmetry operators for a system of 
partial differential equations
given by differential operators which preserve its solution space, see e.g.
\cite{mil}, \cite{olv1}, \cite{olv2}. In the case when the symmetry algebra is
derived from a semi-simple Lie algebra, there is its direct relationship to
the fundamental structural properties of semi-simple Lie algebra including ideal structure 
in its universal enveloping algebra, its adjoint group orbit structure, etc. 
The geometric approach,
allowing to construct these algebraic invariants locally on manifolds with    
a geometric structure, has been successfully completed in the case of Laplace 
operator and conformal differential structure, \cite{eas}, and since then many other 
cases were treated (cf. the bibliography and extensive references therein.)  

The present article is devoted to analogous questions in the case of symplectic
Dirac operator $\Dir_s$ in the real dimension two. The first order differential
symmetries of $\Dir_s$ were already identified with $\mathfrak{sl}(3,\mR)$, cf. \cite{bhs}.
As we shall observe, the algebra of higher symmetries 
leads to projective differential structure and the minimal nilpotent orbit of the 
complexification of $\mathfrak{sl}(3,\mR)$. We shall 
construct the vector space of symmetries as the 
solution space of a class of projectively invariant overdetermined systems
of arbitrarily high order (a substantial difference from the known 
cases with a uniform bound on the orders of higher symmetry differential operators) studied 
on a locally flat projective manifold. The most convenient geometrical 
language allowing a uniform treatment is based on tractor bundles for 
projective parabolic geometries, cf. \cite{cs}. Projectively 
invariant cup product allows to 
introduce associative algebra structure on higher symmetries of $\Dir_s$, and leads 
to an identification of the higher symmetry algebra with an ideal in 
the universal enveloping algebra of 
the complexification of $\mathfrak{sl}(3,\mR)$ given by a quantization of 
the coordinate ring of its
minimal nilpotent orbit. 

\section{Summary of the results}

The present section is a brief and non-technical summary of the content of our article. 
Throughout the article we work over a smooth manifold $M$ equipped with projective geometric structure.
We shall focus particularly on the real dimension $2$, where in addition we assume
the structure group to be the  double cover of $SL(2,\mathbb{R}) \cong Sp(2, \mathbb{R})$. (This  is needed to 
introduce the symplectic spinors.)  We shall use  the Penrose abstract indices  $a,b,\ldots$ for tensor
bundles,  together with the Einstein summation convention understood. 
That is, $\mathcal{E}_a = T^*M$, $\mathcal{E}^a = TM$, $\cE^{ab} = \bigotimes^2 TM$ etc., and we shall use the same 
notation for the spaces of sections. Symmetric tensor products will be denoted by round brackets, e.g.\
$S^2TM = \cE^{(ab)}$. 
We shall use suitable projectively invariant calculus based on the notion of Cartan geometries, 
and we refer to \cite{cs} for details. We use the notation 
$\mN$ for natural numbers $1,2,\ldots$, while $\mN_0$ for $0,1,2,\ldots$.
The abbreviation "lot" denotes lower order terms of a given differential 
operator with a fixed order and symbol, and all considered operators are 
differential. We denote by $\tf$ the trace-free part of the corresponding bundle. 
For example, $\tf(\cE^a{}_b)$ denotes the bundle of trace-free endomorphisms of $TM$.
Though we are interested mainly in $M$ of real dimension $2$, we shall discuss relevant
projectively invariant overdetermined operators in any dimension.

We shall introduce the symplectic Dirac operator in the case of 
real even dimension $2n \geq 2$.
This means that we start with $(M,\omega,\nabla)$, a smooth manifold $M$ equipped with symplectic $2$-form 
$\om_{ab}$ and symplectic covariant derivative $\na$ fulfilling  $\na_a \om_{bc}=0$. The symplectic
$2$-form $\om_{ab}$ allows the identification of $TM$ with $T^*M$. A double cover of the symplectic structure 
gives rise to the associated (the typical fiber being infinite dimensional) 
bundle of symplectic spinors $\cS = \cS_+ \oplus \cS_{-}$ 
induced from the Segal-Shale-Weil representation, \cite{kos}. 
As was already mentioned, by abuse of notation we denote by $\cS$ also the space of sections of $\cS$.
The symplectic Clifford algebra can be realized via symplectic gamma-matrices  
$\ga_a\in \cE_a \otimes \End{\cS}$,  satisfying
\begin{align}\label{gammaidentity1}
\ga_a\ga_b-\ga_b\ga_a=2\om_{ab}.
\end{align}
Here $v^a \ga_a: \cS_\pm \to \cS_\mp$ for any vector field $v^a \in \cE^a$.
Then the symplectic Dirac operator $\Dir_s$ (cf., \cite{hab,cgr}) is defined as 
\begin{align*} 
\omega^{rs} \gamma_s \na_r: \mathcal{S}_\pm \to \mathcal{S}_\mp.
\end{align*}
We refer to \cite{km} for a thorough discussion of  the notion of convenient analysis and 
differential operators acting on sections of infinite dimensional bundles over finite 
dimensional smooth manifolds.

Denoting the space of differential operators acting on $\cS$ by $\Diff(\cS)$, we term $\cO \in \Diff(\cS)$
{\em a higher symmetry} or briefly {\em a symmetry} of $\Dir_s$ provided there exists another differential operator 
$\cO' \in \Diff(\cS)$ such that 
\begin{align}\label{symcond} 
\Dir_s{\fam2 O}={\fam2 O}'\Dir_s. 
\end{align}
Operators of the form ${\fam2 O} = T \Dir_s$ for some differential operator 
$T\in \Diff(\cS)$ are called {\em trivial symmetries}.
Since the composition of symmetries is also a symmetry, the vector space of symmetries is an algebra
denoted $\cA'$ together with the ideal $\cA''\subset\cA'$ of trivial symmetries. 
The aim of our article is to understand 
the quotient-algebra $\cA : =\cA' / \cA''$ in the case of $\Dir_s$. The filtration 
$\Diff^k (\cS) \subseteq \Diff (\cS)$ by order
of operators induces filtration $\cA^k  = \cA \cap \Diff^k(\cS)$ on the algebra $\cA$.
In dimension $2$ and for the flat connection $\na$, 
a direct computation shows that the Lie algebra of first order symmetries of $\Dir_s$ 
is isomorphic to $\mathfrak{sl}(3,\mR)$, \cite{bhs}.
Further, invariance of $\Dir_s$ with respect to the projective geometrical structure 
was recognized in \cite{khs}. This conclusion does not apply in higher dimensions $2n \geq 4$ 
and not much is known in this direction.

Assume the real dimension equals to $2$. Then we can extend the symplectic structure $(M,\om,\na)$ to the projective class of
connections $[\na]$. (Notice the choice $\hat{\na} \in [\na]$ uniquely determines a symplectic form $\hat{\om}$ such that
$\hat{\na}_a \hat{\om}_{ab}=0$.) 
Then the Lie algebra of first order symmetries of $\Dir_s$ is indeed the algebra of infinitesimal projective symmetries
for $(M,[\na])$, \cite{bhs}. Further, we shall need density bundles $\cE(w)$ for a~weight $w \in \mathbb{R}$ 
with standard notational convention
in the projective geometry. That is, the bundle of volume forms is isomorphic to $\cE(-3)$. In analogy to \cite{er},
the bundle of suitably weighted projective frames (or rather its double cover) leads to the associated 
bundle of projective/symplectic spinors denoted by $\cS$ again (see section \ref{spincon} for details).
We put $\cS_\pm(w) := \cS_\pm \otimes \cE(w)$. Then we have a projectively invariant version of the symplectic 
Dirac operator,
\begin{align} \label{sDir}
\Dir_s:= \omega^{rs} \gamma_s \na_r: \mathcal{S}_\pm(-\tfrac34) \to \mathcal{S}_\mp(-\tfrac94), 
\end{align}
in the sense that $\omega^{rs} \gamma_s \na_r$ does not depend on the choice of $\na \in [\na]$ for this
particular choice of weights.

Results in \cite{khs} and \cite{bhs} presume the existence of a flat connection in the projective class $[\na]$.
This corresponds to the  notion of projectively flat structures, see section \ref{projgeom} for a precise definition.
Indeed, characterizing symmetries of $\Dir_s$ on curved projective manifolds is much more complicated
problem, hence  we assume $(M,[\na])$ is projectively flat for the rest of this section.

Principal symbols of operators in $\Diff^k(\cS(w))$ are sections of $S^kTM \otimes \End(\cS)$, 
where $\End(\cS)$ is 
the bundle of symplectic Clifford algebras with the typical fiber isomorphic to \eqref{Cl}. 
Assuming $\cO \in \Diff^k(\cS)$
is a symmetry of $\Dir_s$, we first observe that, modulo $\Dir_s$,  the principal symbol of $\cO$ is a section of 
$S^kTM = (S^kTM\otimes 1)\subseteq S^kTM \otimes \End(\cS)$, see Corollary \ref{pssymalg}.
Moreover, given a section $\si \in S^kTM$, we construct a canonical operator $\cO^\si \in  \Diff^k(\cS(w))$ with
the principal symbol $\si$ for any $w \in \mathbb{R}$, cf.\ \eqref{cQ}. 
Another ingredient are  the differential constraints  for principal symbols  
$\si \in S^kTM$ of symmetries of $\Dir_s$. They are provided
by a family of overdetermined projectively invariant differential operators 
$\Phi_k$ acting on $S^kTM$ and introduced in section \ref{projgeom}, cf.\ \eqref{bgg}. 
Now we are ready to characterize the algebra of  symmetries $\cA$ of $\Dir_s$ as a vector space:

\begin{theorem}
(i) Given $\si \in S^kTM$, the operator $\cO^\si$ is a symmetry of $\Dir_s$ if and only if $\Phi_k(\si)=0$.

(ii) The operator $\cO \in \Diff^k(\cS)$ is a symmetry of $\Dir_s$ if and only if $\cO$ is,
modulo trivial symmetries, of the form
$$
\cO = \cO^{\si_{[k]}} + \cO^{\si_{[k-1]}} + \ldots + \cO^{\si_{[0]}}: 
\cS(-\tfrac34) \to \cS(-\tfrac34)
$$
where $\si_{[i]} \in S^iTM$ and $\Ph_i(\si_{[i]})=0$ for all $i=0, \ldots, k$.
This in particularly means that $\cO$, modulo a trivial symmetry,  preserves the decomposition
$\cS(-\tfrac34) = \cS_+(-\tfrac34) \oplus \cS_-(-\tfrac34)$.

\end{theorem}

This theorem follows from Theorem \ref{mainsymbols} and the discussion before this theorem.

The zero order symmetries are given by multiplication with a constant zero order 
differential operators, i.e.\ $\cA^0 \cong \mathbb{R}$. Further,
we write explicitly the first order symmetries (modulo constants): an operator 
$\cO \in \cA^1/\cA^0$ is a symmetry of $\Dir_s$ if and only if
\begin{align} \label{1st}
& \cO  = \si^r\nabla_r + \tfrac{1}{4}(\nabla_r \si^s) \ga^r\ga_s + \tfrac{1}{3}(\nabla_r v^r), \\
& \cO' = \si^r\nabla_r + \tfrac{1}{4}(\nabla_r \si^s) \ga^r\ga_s + \tfrac{5}{6}(\nabla_r v^r)
\end{align}
modulo trivial symmetries. Here the symbol $\si^a$ of $\cO, \cO'$ satisfies 
projectively invariant overdetermined differential system $\Phi_1(\si)=0$:
$$
\Phi_1(\si) = \tf (\na_{(a} \na_{b)} \si^c + \Ric_{ab} \si^c),
$$
with $\Ric_{ab}$ the Ricci tensor of $\na$. That is, principal symbols of first order symmetries
(modulo constants $\cA^0$)
are exactly infinitesimal projective symmetries isomorphic to the Lie algebra $\mathfrak{sl}(3,\mathbb{R})$.
Of course, this is expected 
since $\Dir_s$ is a projectively invariant differential operator, cf. \cite{khs}.

The algebra structure on $\cA$ is the content of our second main result proved in Section \ref{sec:5}.

\begin{theorem}
The algebra of symmetries $\cA$ of $\Dir_s$ is isomorphic to the quotient of the tensor algebra
$\bigoplus\limits_{k=0}^\infty\otimes^k\big(\mathfrak{sl}(3,{\mathbb R})\big)$
by a two sided ideal generated by quadratic relations
\begin{align*}
I\otimes\bar{I}-I\boxtimes\bar{I}-\tfrac{1}{2}[I,\bar{I}]+\tfrac{1}{32}\langle I,\bar{I}\rangle_K 
\end{align*}
for $I, \bar{I} \in \mathfrak{sl}(3,{\mathbb R})$. Here $[\,,\,]$ and
$\langle \, ,\, \rangle_K$ are the Lie bracket and the Killing form on $\mathfrak{sl}(3,{\mathbb R})$,
respectively, and  $\boxtimes$ denotes the Cartan product.

Equivalently, $\cA$ is the quotient of the universal enveloping 
algebra $U(\mathfrak{sl}(3,{\mathbb R}))$ by a two sided ideal generated by quadratic 
relations 
\begin{align*}
I\bar{I}+\bar{I}I-2I\boxtimes\bar{I}+\tfrac{1}{16}\langle I,\bar{I}\rangle_K.
\end{align*}
\end{theorem}

The ideal defined on the last display corresponds
to a completely prime primitive ideal having as its associated 
variety the minimal nilpotent orbit of the complexification of
$\mathfrak{sl}(3,\mR)$, cf. \cite{jos,tor,gar,ess}.
The present article can be regarded as a geometric construction
of this exceptional ideal.

\section{Symplectic spinors and projective geometry} \label{sec:3}

\subsection{Symplectic spinors} \label{ss}

We present basic algebraic preliminaries related to 
the construction and realization of symplectic spinors, cf. \cite{hab,cgr}.
The metaplectic Lie group $Mp(2n,\mR)$ is the non-trivial
double covering of the symplectic Lie group $Sp(2n,\mR)$ of 
automorphisms of the standard symplectic space 
$(\mR^{2n},\omega)$. Their Lie
algebras are denoted by $\mathfrak{mp}(2n,\mR)$ and $\mathfrak{sp}(2n,\mR)$,
respectively.   
The symplectic spinor representation for the metaplectic Lie algebra 
$\mathfrak{mp}(2n,\mR)$  is given by two simple metaplectic modules
of the Segal-Shale-Weil representation, modeled on the vector space
of polynomials on a Lagrangian subspace $\mR^n$ of $(\mR^{2n},\omega)$.  
The ${\mC}$-algebra of endomorphisms of their direct sum  
is called symplectic Clifford algebra (or, the Weyl algebra) and is
isomorphic to the quotient of the tensor algebra $T(\mR^{2n})$ by a 
two sided ideal $I$:
\begin{align} \label{Cl}
\begin{aligned}
& Cl_s(\mR^{2n},\omega):= T(\mR^{2n})/\langle I\rangle, \\
& I=\{e_ie_j-e_je_i=2\omega_{ij}|\, i,j=1,\ldots ,2n, e_i, e_j\in\mR^{2n}\}.
\end{aligned}
\end{align}
We denoted by $e_1, \ldots , e_{2n}$ a basis of $\mR^{2n}$ and
$\omega_{ab}=\omega(e_a,e_b)$. 

The inverse of the symplectic $2$-form $\omega_{ab}$ is denoted $\omega^{ab}$,
and we use the convention $\omega^{ja}\omega_{jb}=\delta^{a}{}_{b}$ with the 
summation over $j=1,\dots ,2n$ understood. The 
composition of the isomorphisms $\omega_{ab}: TM\to T^* M$, 
$v_a\mapsto \omega^{aj}v_j$, and
$\omega^{ab}: T^* M\to TM$, $v^a\mapsto \omega_{ja}v^j$, is then the 
identity endomorphism of $TM$:
\begin{align}
v_a\mapsto \omega^{aj}v_j\mapsto\omega_{ka}\omega^{kj}v_j=v_a .
\end{align}
This is equivalent to $\omega^{a}{}_{b}=-\delta^{a}{}_{b}$,
$\omega_{a}{}^{b}=\delta_{a}{}^{b}$. As for the scalar product
induced by the symplectic form, we have $v^jw_j=\omega^{jk}v_k\omega_{lj}w^l$
so that $v_jw^j=-v^jw_j$.
 
We denote basis elements $e_1, \ldots , e_{2n}$ of $\mR^{2n}$ by $\gamma_1,\ldots , 
\gamma_{2n}$ when regarded as generators of $Cl_s(\mR^{2n},\omega)$, i.e.,
$\gamma_i\in T^*M\otimes Cl_s(\mR^{2n},\omega)$ for all $i=1, \ldots , 2n$. Then
\begin{align*}
\gamma_i\gamma_j-\gamma_j\gamma_i=2\omega_{ij},\quad 
\gamma_{[i}\gamma_{j]}=\omega_{ij},
\end{align*}
where their traces are 
\begin{align}\label{gammaidentity2}
\gamma^k\gamma_k=-2n,\quad
\gamma_k\gamma^k=2n, 
\end{align}
and the commutator with symmetrized product of $\gamma$'s is
\begin{align}\label{gammaidentity3}
\gamma^a\gamma^{(a_1}\ldots\gamma^{a_{j})}-\gamma^{(a_1}\ldots\gamma^{a_{j})}\gamma^a=
2j\,\omega^{a(a_1}\gamma^{a_2}\ldots\gamma^{a_{j})}
\end{align}
for any $j\in \mN$. As already indicated, in the present article we are mostly
concerned with the case $n=1$.

We have the inclusion $\frak{sp}(2n,\mathbb{R}) \cong S^2 \mathbb{R}^{2n} \subseteq  Cl_s(\mR^{2n},\omega)$ realized by 
quadratic elements in the generators $\mathbb{R}^{2n} \subseteq  Cl_s(\mR^{2n},\omega)$. 
This promotes to the bundle level as follows:
the section $F^{ab} \in \cE^{(ab)}$ acts on $\ph \in \cS$ by
$$
\ph \mapsto F (\ph) = -\tfrac14 F^{ab} \gamma_a \gamma_b \ph,
$$
where the coefficient $\frac14$ ensures this  indeed corresponds to the Lie algebra representation. 
If $n=1$, this can be expressed as the action of endomorphisms $F^a{}_b \in \tf (\cE^a{}_b)$, 
\begin{equation} \label{action}
\ph \mapsto \tfrac14 F^a{}_b \gamma_a \gamma^b \ph.
\end{equation}

The symplectic Dirac operator $\Dir_s$ acts on symplectic spinors induced from the 
half-integral Segal-Shale-Weil $\smash{\widetilde{SL}}(2,\mR)$-representation, \cite{kos}, \cite{hab}.
It was introduced in the seminal work \cite{kos} for the purpose of geometric 
quantization on any symplectic manifold with a metaplectic structure.
In \eqref{sDir}, the algebraic map $TM\otimes\mathcal{S}^\pm\to\mathcal{S}^\mp$ 
follows from the 
embedding of $TM$ into the bundle of symplectic Clifford algebras, and $\nabla$
denotes the induced symplectic covariant derivative on $\mathcal{S}^\pm$.

\subsection{Projective geometry and tractor calculus}
\label{projgeom}

Projective structure on a smooth manifold $M$ of
real dimension greater than equal to $2$ is a class 
$[\nabla]$ of torsion-free volume preserving connections, which define 
the same family of unparametrized geodesics. A connection is projectively flat if and only if it is locally
equivalent to a flat connection, which means there exists a local isomorphism
with the flat model of $n$-dimensional homogeneous projective geometry on $\mathbb{RP}^n$
equipped with the flat projective structure given by the absolute
parallelism.
The homogeneous model of projective geometry in the real
dimension $2$ is $\mathbb{RP}^2\simeq G/P$, where $G\simeq SL(3,\mR)$ 
and $P\subset G$ the parabolic
subgroup stabilizing the line $[v]\in\mR^3$ generated by a non-zero vector $v$ in
the defining representation $\mR^3$ of $G$.
The construction of associated vector
bundles induced from half integral modules of $P$ 
(e.g.,\ the simple metaplectic submodules
of the Segal-Shale-Weil representation for the metaplectic group) 
on $G/P$ requires the double (universal) cover
$\smash{\widetilde{G}}=\smash{\widetilde{SL}}(3,\mR)$ and its parabolic subgroup 
$\smash{\widetilde{P}}$. The Lie
group $\smash{\widetilde{SL}}(3,\mR)$ acts transitively on $S^2\simeq \mC\mP^1$,
the double (universal) cover of $\mathbb{RP}^2$, with parabolic stabilizer
$\smash{\widetilde{P}}=(GL(1,\mR)_+\times \smash{\widetilde{SL}}(2,\mR))\ltimes\mR^2$.
The double (universal) cover 
$\smash{\widetilde{SL}}(3,\mR)/\smash{\widetilde{P}}\simeq S^2\simeq \mC\mP^1$ 
is a symplectic manifold,
while $\mR\mP^2$ is non-orientable and hence not symplectic.

Further, we follow conventions for projective structures as in \cite{beg}. 
Recall the notation $\cE(w)$ for density   bundles, $w\in\mC$.
The difference between two connections  $\na, \wh{\na} \in [\na]$ 
for a given projective structure $[\na]$
is controlled by a one-form $\Upsilon_a:=\nabla_a\log(f)$,
 $f\in \cE\equiv C^\infty(M)$. Specifically, we have
\begin{align}
&\hat{\na}_a \alpha = \na_a \alpha + w\Up_a \alpha, 
\quad \alpha \in \cE(w), \notag \\
&\hat{\na}_a V^b = \na_a V^b + \Up_a V^b + \Up_c V^c \de_a^b,
\quad  V^b \in \cE^b, \label{trans} \\
&\hat{\na}_a \mu_b = \na_a \mu_b - \Up_a \mu_b - \Up_b \mu_a, 
\quad \mu_a \in \cE_a. \notag
\end{align}

The curvature tensor $R_{ab}{}^c{}_d$ of $\na$ is defined by 
$(\na_a \na_b - \na_b \na_a) V^c = R_{ab}{}^c{}_p V^p$ and it decomposes as
$$
R_{ab}{}^c{}_d = W_{ab}{}^c{}_d + 2\de_{[a}{}^c \Rho_{b]d},
$$
where the Schouten tensor $\Rho_{ab}$ is symmetric. Note the Ricci tensor equals to 
$\Ric_{ab} = (n-1)\Rho_{ab}$.
Here $W_{ab}{}^c{}_d$ is projectively invariant (and irreducible) Weyl 
tensor, and the Cotton-York tensor $Y_{abc} := 2\na_{[a}\Rho_{b]c}$ is projectively invariant in dimension $2$.
The projective structure $(M,[\na])$ is locally flat if and only if $W_{ab}{}^c{}_d=0$ (in dimension $2n \geq 3$)
or $Y_{abc} =0$ (for $n=1$). There is a $\na$-parallel volume form $\epsilon \in \La^{2n} T^*M$ and
we have the projective volume form $\bfep \in \La^{2n}T^*M(2n+1)$ which is parallel for any $\na \in [\na]$.
Any choice $\na \in [\na]$ in dimension $2$ gives a symplectic structure $\om = \epsilon$ and we shall 
use the notation $\om$ in this dimension. 
Similarly, we have the weighted version $\bfom_{ab} = \bfep_{ab} \in \cE_{[ab]}(3)$
with its dual $\bfom^{ab} \in \cE^{[ab]}(-3)$

We shall write sections of the standard projective tractor bundle 
$\cE^A = \cE^a(-1) \lpl \cE(-1)$, resp.\ its dual 
$\cE_A = \cE(1) \lpl \cE_a(1)$ using the injectors $Y^A$, $X^A$,
resp.\ $Y_A$, $X_A$ as 
\begin{align} \label{injectorsfundamental}
\begin{pmatrix} 
\si^a \\ \rh
\end{pmatrix}
= Y^A{}_a \si^a + X^A \rh \in \cE^A,
\ \ \text{resp.} \ \
\begin{pmatrix} 
\nu \\ \mu_a
\end{pmatrix}
= Y_A \nu + X_A{}^a \mu_a \in \cE_A.
\end{align}
These splittings of $\cE^A$ and $\cE_A$ are determined by choices of projective
connections and we call them projective splittings.
The change of the splitting under the change of connection parametrized by
$\Upsilon_a \in \cE_a$ is
\begin{align*}
&\widehat{
\begin{pmatrix} 
\si^a \\ \rh
\end{pmatrix}
} = 
\begin{pmatrix} 
\si^a \\ \rh - \Upsilon_a \si^a  
\end{pmatrix},
\ \ \mbox{i.e.} \ \
\hat{Y}^A{}_a = Y^A{}_a + X^A \Upsilon_a,\ 
\hat{X}^A = X^A
\quad \text{and} \\
&\widehat{
\begin{pmatrix} 
\nu \\ \mu_a
\end{pmatrix}
} = 
\begin{pmatrix} 
\nu \\ \mu_a + \Upsilon_a \nu  
\end{pmatrix},
\ \ \mbox{i.e.} \ \
\hat{Y}_A = Y_A - X_A{}^a \Upsilon_a,\ 
\hat{X}_A{}^a = X_A{}^a.
\end{align*}
That is,  $X^A \in \cE^A(1)$, $X_A{}^a \in \cE_A{}^a(-1)$ are invariant
and $Y^A{}_a \in \cE^A{}_a(1)$, $Y_A \in \cE_A(-1)$ depend on the choice
of the projective scale. 
We assume the normalization of these in which 
$Y_A X^B + X_A{}^c Y^B{}_c = \de_A{}^B$, i.e., \ $Y_C X^C = 1$ and 
$X_C{}^a Y^C{}_b = \de^a{}_b$.

The normal covariant derivative is given by
\begin{align*}
&\na_c
\begin{pmatrix} 
\si^a \\ \rh 
\end{pmatrix} = 
\begin{pmatrix} 
\na_c \si^a + \rh \de_c{}^a \\ \na_c \rh - P_{cp} \si^p
\end{pmatrix} 
\quad \text{and} \quad
\na_c 
\begin{pmatrix} 
\nu \\ \mu_a 
\end{pmatrix} = 
\begin{pmatrix} 
\na_c \nu - \mu_c  \\ \na_c \mu_a + P_{ca} \nu
\end{pmatrix}, \quad \text{i.e.}\\
&\na_c^{} Y^A{}_a = -X^A \Rho_{ca},\ \na_c X^A = Y^A{}_c, \ 
\na_c^{} Y_A = X_A{}^a \Rho_{ca},\ \na_c X_A{}^a = -Y^A \de_c{}^a,
\end{align*}
and its curvature $\Omega$ has the form
$$
\Omega_{ab}{}^E{}_F = Y^E{}_e X_F{}^f W_{ab}{}^e{}_f - X^E X_F{}^f Y_{abf} \in
\cE_{[ab]} \otimes \lA_{\mathfrak{d}}.
$$
That is, $\lA_{\mathfrak{d}} := \tf (\cE^E{}_F)$ is the projective adjoint 
tractor bundle, hence the curvature action on $\cE_C$ is
$(\na_a\na_b - \na_b\na_a) F_C =  - \Omega_{ab}{}^D{}_C F_D$. 

We shall be interested in symmetric tensor powers of $\lA_{\mathfrak{d}} \subseteq {\mathcal E}^A{}_B$.
Injectors for the bundle ${\mathcal E}^A{}_B$ defined as
\begin{align}
{\mathbb Y}_{a}{}^A{}_{B}:=Y_a{}^AY_B, \quad  & {\mathbb Z}_a{}^{b}{}^A{}_{B}:=Y_a{}^AX^b{}_B,
\nonumber \\
{\mathbb W}^{A}{}_{B}:=X^AY_B, \quad & {\mathbb X}^{b}{}^A{}_{B}:=X^AX^b{}_B,
\end{align}
are acted upon by the covariant derivative,
\begin{align}
& \nabla_c{\mathbb Y}^{A}{}_{a}{}_B=-P_{ca}{\mathbb W}^{A}{}_{B}
+P_{cb}{\mathbb Z}_a{}^{b}{}^A{}_{B}, \nonumber \\ 
& \nabla_c{\mathbb Z}_a^{}{b}{}^A{}_{B}=-P_{ca}{\mathbb X}^{b}{}^A{}_{B}
-{\mathbb Y}_a{}^{A}{}_{B}\delta_c{}^{b},\nonumber \\
& \nabla_c{\mathbb W}^{A}{}_{B}={\mathbb Y}_c{}^{A}{}_{B}+
P_{ac}{\mathbb X}^{a}{}^A{}_{B}, \nonumber \\ 
& \nabla_c{\mathbb X}^{b}{}^A{}_B=\,\,{\mathbb Z}_c{}^b{}^A{}_{B}
-{\mathbb W}^{A}{}_{B}\delta_c{}^{b} \label{covderinj}.
\end{align}
As for the tensor product of injectors, we have  
\begin{align}\label{injoperpartscalprod}
& {\mathbb Z}_b{}^a{}^B{}_{R}{\mathbb Z}_d{}^c{}^R{}_{C} = \,
{\mathbb Z}_b{}^c{}^B{}_C\delta_d{}^{a}, & &
{\mathbb Z}_r{}^{b}{}^{R}{}_{B}{\mathbb X}^{s}{}^{A}{}_{R} = 
\delta_r{}^{s}{\mathbb X}^{b}{}^{A}{}_{B},
\nonumber \\
& {\mathbb X}^b{}^B{}_{R}{\mathbb Y}_d{}^{R}{}_{C}= \, {\mathbb W}^{B}{}_{C}\delta_d{}^{b}, 
 & & {\mathbb X}^r{}^R{}_{B}{\mathbb W}^{A}{}_{R} =  {\mathbb X}^r{}^A{}_{B},
\nonumber \\
& {\mathbb X}^b{}^R{}_{C}{\mathbb Y}_d{}^{B}{}_{R} = \, {\mathbb Z}_d{}^{b}{}^B{}_{C},
 & & {\mathbb Z}_r{}^{s}{}^{A}{}_{R}{\mathbb Y}_a{}^{R}{}_{B} =  
\delta^{s}{}_{a}{\mathbb Y}_r{}^{A}{}_{B},
\end{align}
and their contractions are
\begin{align}\label{injoperscalprod}
{\mathbb Y}_a{}^{A}{}_{B}{\mathbb X}^b{}^B{}_{A}=\delta_a{}^{b},\quad
{\mathbb Z}_a{}^{b}{}^A{}_{B}{\mathbb Z}_c{}^{d}{}^B{}_{A}=\delta_a{}^{d}\delta_c{}^{b},\quad
{\mathbb W}^{A}{}_{B}{\mathbb W}^{B}{}_{A}=1,
\end{align}
respectively. By \eqref{injectorsfundamental}, we write sections of  $\lA_{\mathfrak{d}}$ as 
\begin{align*}
I^A{}_B := \left( \begin{array}{c}  v^a  \\ \mu^a{}_{b} \ \ \mid \ \  \varphi \\  \rho_a \end{array} \right)
= v^c{\mathbb Y}_c{}^{A}{}_{B} + \mu^d{}_{c}{\mathbb Z}_d{}^{c}{}^A{}_{B}
+\varphi{\mathbb W}^{A}{}_{B}+\rho_c{\mathbb X}^{c}{}^A{}_{B},
\end{align*}
where the trace-free condition $\mu^a{}_{a}+\varphi=0$, equivalent to 
$I^A{}_A=0$, is implied by  
$\mathfrak{sl}(3,\mathbb R)$-irreducibility. The set of formulas 
\eqref{covderinj} can be rewritten as
\begin{align}
&\nabla_c\left( \begin{array}{c} v^a  \\  \mu^a_{}{b} \ \ \mid \ \ \varphi \\  \rho_a   \end{array} \right)
= \nabla_c\big (v^d{\mathbb Y}_d{}^{A}{}_{B} + \mu^d{}_{e}{\mathbb Z}_d{}^{e}{}^A{}_{B}
+\varphi{\mathbb W}^{A}{}_{B}+\rho_d{}^d{\mathbb X}^{d}{}^A{}_{B}\big )
 \nonumber \\ \nonumber
&= \left( \begin{array}{c} 
  \nabla_cv^a-\mu^a{}_{c}+\varphi\delta^a{}_{c}  \\
  \nabla_c\mu^a{}_{b}+v^aP_{cb}+\rho_b\delta^a{}_{c} \qquad \mid \qquad  \nabla_c\varphi-v^aP_{ca}-\rho_c \\
  \nabla_c\rho_a+\varphi P_{ac}-\mu^b{}_{a}P_{cb} 
\end{array} \right) , \\ \label{tractcovderexp}
\end{align}
where $\mu^a{}_{a}+\varphi=0$ is preserved by the tractor 
covariant derivative. 
If  $I^A{}_{B} \in \lA_{\mathfrak{d}}$ is covariantly constant,
\begin{eqnarray}\label{covconstsect}
\nabla_cI^A{}_{B}=
\nabla_c\left( \begin{array}{ccc}
 & v^a &  \\
 \mu^a{}_{b} & & \varphi \\
 & \rho_a &  
\end{array} \right)
=0,
\end{eqnarray} 
then the top slot $v^a \in \mathcal{E}^a$ satisfies 
the projectively invariant equation
\begin{equation} \label{bgginf}
\tf \bigl( \nabla_{(a} \nabla_{b)} v^c + \Rho_{ab} v^c \bigr)=0.
\end{equation}
This follows from \eqref{tractcovderexp} after a short computation.
Moreover, this equation is equivalent to $\na_c I^A{}_{B}=0$ on projectively flat manifolds.
If we use $\bfom_{ab}$ and consider $v_a \in \cE_a(3)$, the previous display converts into
$\nabla_{(a} \nabla_{b} v_{c)} + \Rho_{(ab} v_{c)} =0$ due to specific properties  
of projective geometry in the real dimension $2$.
The equation \eqref{bgginf} is the special case of the projectively invariant first BGG operator
on symmetric tensors,
\begin{align} \label{bgg}
\begin{aligned}
\Phi: \,&\cE^{(a_1 \ldots a_k)} \to \tf(\cE_{(c_0 \ldots c_{k})}{}^{(a_1 \ldots a_k)} \bigr)
\cong \cE^{(c_0 \ldots c_k a_1 \ldots a_k))}(-3(k\!+\!1)), \\
&\, \si^{a_1 \ldots a_k} \mapsto \tf \bigl( \na_{(c_0} \ldots \na_{c_{k})} \si^{a_1 \ldots a_k} \bigr) + \mathrm{lot} .
\end{aligned}
\end{align}
The following lemma is easily verified.
\begin{lemma}\label{lemma31}
The projection  
$\Pi: {\mathcal E}^A{}_{B} \rightarrow \, {\mathcal E}^a$ defined on sections by
\begin{align}
\left( \begin{array}{ccc}
 & v^a &  \\
 \mu_a{}^{b} & & \varphi \\
 & \rho_a &  
\end{array} \right)
& \mapsto \, v^a,
\end{align} 
has projectively invariant differential splitting  
$\Si: {\mathcal E}^a \rightarrow {\mathcal E}^A{}_{B}$ given by
\begin{align}\label{splittingsl3adjoint}
v^a \mapsto \big(\Sigma v^a\big)^{A}{}_{B} = & \,\,
{\mathbb Y}_a{}^{A}{}_{B}v^a
+{\mathbb Z}_a{}^b{}^A{}_B\big((\nabla_bv^a)_o+\tfrac{1}{6}\delta_b{}^{a}\nabla_cv^c\big)
\nonumber \\
& -\tfrac{1}{3}{\mathbb W}^{A}{}_{B}\nabla_cv^c
-{\mathbb X}^a{}^A{}_B(\tfrac{1}{3}\nabla_a\nabla_b+P_{ab})v^b .
\end{align} 
\end{lemma}

Apart from the adjoint tractor bundle $\lA_{\mathfrak{d}}$, we shall need the bundles
$\tf\bigl( \cE^{(A_1 \ldots A_k)}{}_{(B_1 \ldots B_k)})$. Analogously as above, we have the projection
$\Pi$ and its differential splitting $\Si$ (by abuse of notation we use the same letters as in Lemma
\ref{lemma31})
\begin{align} \label{split}
\begin{aligned}
& \Pi: \tf \bigl( \cE^{(A_1 \ldots A_k)}{}_{(B_1 \ldots B_k)} \bigr) \to \cE^{(a_1 \ldots a_k)},  \\
& \Si: \cE^{(a_1 \ldots a_k)} \to \tf \bigl( \cE^{(A_1 \ldots A_k)}{}_{(B_1 \ldots B_k)} \bigr),
\end{aligned}
\end{align}
i.e.\ $\Pi \circ \Si = \id$. The existence of projectively invariant splitting $\Si$ follows from 
the BGG machinery. Moreover, $\Si$ is unique in the projectively flat case and, under certain normalization
condition, also in the curved case. This is known as the 
{\em BGG splitting operator} and henceforth will be our choice of $\Si$.
It follows from the BGG machinery that $\Si$ has the following essential property:

\begin{proposition} \label{prol}
Assume $(M,[\na])$ is a projectively flat structure and let $\si^{a_1 \ldots a_k} \in \cE^{(a_1 \ldots a_k)}$.
Then $\Ph(\si) = 0$ if and only if $\na \Si(\si)=0$. That is, there is a bijective correspondence
$$
\{\si^{a_1 \ldots a_k} \in \cE^{(a_1 \ldots a_k)} \mid \Ph(\si) = 0 \} \, \stackrel{1:1}{\longleftrightarrow} \,
\boxtimes^k \mathfrak{sl}(3,\mathbb{R}).
$$
\end{proposition}
We recall the notation $\boxtimes^k \mathfrak{sl}(3,\mathbb{R})$ for the 
$k$-th Cartan power of $\mathfrak{sl}(3,\mathbb{R})$
in $\otimes^k \mathfrak{sl}(3,\mathbb{R})$, $k\in\mathbb{N}$.

\subsection{Projective connection on symplectic spinors}
\label{spincon}

From now on we assume $n=1$ and use the symplectic form $\bfom_{ab} = \bfep_{ab} \in \cE_{[ab]}(3)$. 
Then the symplectic spinor bundle can be realized and described in projective geometry similarly to the spinor bundle 
in conformal geometry, cf. \cite{er}. It is convenient to induce bundles
from symplectic frames of $\cE_a(\frac32)$ where the pairing $\langle \mu, \nu \rangle  = \bfom^{ab} \mu_a\nu_b$
is projectively invariant for $\mu_a, \nu_a \in \cE_a(\frac32)$. 
We have a weighted version of symplectic gamma matrices $\bfga_a$,
$$
\bfga_a \bfga_a  - \bfga_b \bfga_a = 2 \bfom_{ab}, \quad \bfga_a \in \cE_a \otimes (\End \cS)(\tfrac32),
$$
cf.\ \eqref{gammaidentity1}, which generate the bundle of weighted Clifford algebras.
The projective transformation \eqref{trans} implies
$$
\wh{\na}_a \mu_b = \na_a \mu_b + \tfrac12 \Up_{a} \mu_{b}^{} - \Up_b \mu_a = 
\na_a \mu_b  - \Ga_a{}^c{}_b \mu_c, \quad \text{where} \quad \Ga_{abc} = \bfom_{a(b} \Up_{c)}.
$$
Considering $\Ga_a{}^c{}_b$ as a one form valued in the bundle of Lie algebras $\mathfrak{sl}(2,\mathbb{R})$ 
and using the action \eqref{action}, we obtain
$$
\wh{\na}_a \ph = \na_a \ph - \tfrac14 \Ga_a{}^c{}_b \bfga_c \bfga^b \ph
= \na_a \ph + \tfrac14 \Up_a \ph - \tfrac14 \bfga_a \Up_s \bfga^s \ph, \qquad \ph \in \cS.
$$
By \eqref{trans}, we obtain the weighted version
\begin{align}\label{covderonsymplspinors}
\hat{\nabla}_a\varphi=\nabla_a\varphi+(w+\tfrac{1}{4})\Upsilon_a\varphi
-\tfrac{1}{4}\bfga_a \Upsilon_s \bfga^s\varphi, \quad \varphi\in{\cal S}(w).
\end{align}
\begin{lemma}\label{sdirinvariance}
The symplectic Dirac operator $\Dir_s=\bfga^a\nabla_a$ acting on $\cS(w)$ is projectively
invariant differential operator if and only if $w=-\frac34$,
\begin{align}    
\Dir_s: {\mathcal S}(-\tfrac34)\to {\mathcal S}(-\tfrac94).
\end{align}
\end{lemma}
\begin{proof}
This results from the contraction of \eqref{covderonsymplspinors} by $\bfga^a$. 
\end{proof}

Next we observe there is a spinor version of the  Thomas $D$-operator
$D_A: \cE(w) \to \cE_A(w-1)$, \cite{beg}.  Its analogue acting on symplectic spinors is given as follows:

\begin{theorem}
The tractor $D$-operator acting on $\cS(w)$ and defined by
\begin{align} \label{tractorDoper}
& \hspace{2cm} D_A :{\mathcal S}(w)\to {\mathcal S}(w-1)\otimes{\mathcal E}_A,
\nonumber \\
& D_A=(4w+3)(w+\tfrac{1}{4})Y_A+(4w+3)X^b{}_A\nabla_b+\bfga_bX^b{}_A \Dir_s,
\end{align}
is a projectively  invariant first order differential operator.
\end{theorem}
The projective invariance of $D_A$ is a straightforward consequence of
the transformation property \eqref{covderonsymplspinors}. In the case $w=-\frac34$, $D_A$
reduces to the invariant bottom slot of the standard tractor bundle and hence $\Dir_s$ is projectively invariant 
operator for this value of $w$ in a full agreement with Lemma \ref{sdirinvariance}. Another interesting case 
corresponds to $w=-\frac14$ for which $D_A$ reduces to the projectively invariant 
symplectic twistor  differential operator acting on spinors,
$$
\nabla_a+\tfrac{1}{2}\bfga_a\Dir_s: \cS(-\tfrac14) \to \cE_a \boxtimes \cS(-\tfrac14).
$$

\begin{theorem}
The tractor operator ${\mathbb D}$ acting on $\cS(w)$ and defined by
\begin{align} \label{doubleD}
{\mathbb D}^B{}_{A} & :\quad {\fam2 S}
\rightarrow \tf ( {\cal E}^B{}_{A}) \otimes \cS(w),
\\ \nonumber
& \ph \mapsto {\mathbb D}^B{}_{A} \ph:=
\left( \begin{array}{c} 0  \\ 
\frac{1}{4}\bfga^a\bfga_b \ph +(-\frac{w}{3}+\tfrac{1}{4})\delta^a{}_{b} \ph \quad \mid \quad \frac{2w}{3} \ph \\
\nabla_a \ph 
\end{array} \right)
\end{align}
or in terms of injectors by
\begin{align}
{\mathbb D}^B{}_{A} = Y_a{}^AX^b{}_B\big(\tfrac{1}{4}\bfga^a\bfga_b 
+(-\tfrac{w}{3}+\tfrac{1}{4})\delta_{b}{}^{a}\big)
+\tfrac{2w}{3}X^AY_B
+X^AX^b{}_B\nabla_b,
\end{align}
is a projectively invariant first order differential operator. We have the 
commutation relation of invariant differential  operators 
\begin{align}\label{commutdsprolong}
\Dir_s\, {\mathbb D}^B{}_{A} =
{\mathbb D}^B{}_{A} \, \Dir_s: \cS (-\tfrac34) \to \tf (\cE^B{}_A) \otimes   \cS (-\tfrac94) 
\end{align}
\end{theorem}

\begin{proof}
The first claim, related to projective invariance, 
is straightforward and follows from the transformation properties 
of injectors $Y_a{}^A, X^b{}_B, X^A, Y_B$ and $\nabla$.
The second claim \eqref{commutdsprolong} follows from an explicit 
computation.  Then 
\begin{align*} 
\Dir_s\, & {\mathbb D}^B{}_A  =
(-P_{ca}{\mathbb X}^b{}^{A}{}_B-{\mathbb Y}_a{}_{B}{}^A\delta_c{}^b)
(\tfrac{1}{4}\bfga^c\bfga^a\bfga_b +\tfrac{1}{2}\bfga^c\delta_{b}{}^{a})\nonumber\\
&-\tfrac{1}{2}\bfga^c({\mathbb Y}_{c}{}^A{}_B+P_{ac}{\mathbb X}^a{}^{A}{}_B)\nonumber 
+\bfga^c({\mathbb Z}_c{}^{b}{}^A{}_B-{\mathbb W}^A{}_{B}\delta_c{}^b)\nabla_b
\nonumber \\
& + {\mathbb Z}^{b}{}_c{}^A{}_B(\tfrac{1}{4}\bfga^c\bfga^a\bfga_b +
\tfrac{1}{2}\bfga^c\delta_{b}{}^{a})\nabla_c \nonumber 
-\tfrac{1}{2}{\mathbb W}^A{}_B\bfga^c\nabla_c+{\mathbb X}^b{}^{A}{}_B\bfga^c\nabla_c\nabla_b\nonumber \\
& ={\mathbb Y}_{a}{}^A{}_B(-\frac{1}{4}\bfga^c\bfga^a\bfga_c -
\tfrac{1}{2}\bfga^a-\tfrac{1}{2}\bfga^a) \nonumber 
+{\mathbb Z}^{b}{}_a{}^A{}_B(\bfga^a\nabla_b+\tfrac{1}{4}\bfga^c\bfga^a\bfga_b\nabla_c
+\tfrac{1}{2}\bfga^c\delta_b{}^a\nabla_c)\nonumber \\
& +{\mathbb W}^A{}_B(-\bfga^c\delta_c{}^b\nabla_b-\tfrac{1}{2}\bfga^c\nabla_c)\nonumber 
+ {\mathbb X}^b{}^{A}{}_B( - \tfrac{1}{4} P_{ca}\bfga^c\bfga^a\bfga_b - \tfrac{1}{2} P_{ca}\bfga^c\delta_b{}^a
\nonumber \\
& -\tfrac{1}{2}\bfga^cP_{ac}+\bfga^c\nabla_c\nabla_b).
\end{align*}
The identities \eqref{gammaidentity1} and \eqref{gammaidentity2} imply
the coefficient of ${\mathbb Y}_{a}{}^A{}_B$ is zero. The operator
of ${\mathbb W}^A{}_B$ quotients from the right through $\Dir_s$. 

Let us now discuss the operator coefficient of 
${\mathbb X}^b{}^{A}{}_B$. The curvature of $\nabla$
acts on symplectic spinors as 
$\nabla_a\nabla_b-\nabla_b\nabla_a=-\frac{1}{4}R_{ab}{}^c{}_d\bfga_c\bfga^d$. 
In the real dimension $2$ holds the curvature identity 
$R_{ab}{}^c{}_d=\delta_a{}^cP_{bd}-\delta_b{}^cP_{ad}$, 
hence  
\begin{align*}
\bfga^cR_{cb}{}^e & {}_f\bfga_e\bfga^f  =
\bfga^c(\delta_c{}^eP_{bf}-\delta_b{}^eP_{cf})\bfga_e\bfga^f 
\nonumber \\
& = -2\bfga^fP_{bf} - \bfga^c\bfga_b\bfga^fP_{cf}
\nonumber 
 = -4\bfga^fP_{bf} - \bfga^c\bfga^f\bfga_bP_{cf}.
\end{align*}
The substitution for $\bfga^c\nabla_c\nabla_b$ results in 
the contribution ${\mathbb X}^{bA}{}_B\nabla_b\Dir_s$.

The operator coefficient of ${\mathbb Z}^{b}{}_a{}^A{}_B$ can be rewritten as
$\frac{1}{2}\bfga^a\nabla_b-\frac{1}{2}\bfga_b\nabla^a
+\frac{1}{4}\bfga^a\bfga_b\Dir_s
+\frac{1}{2}\delta^a{}_b\Dir_s$, hence in the real dimension $2$ we get
\begin{align*}
\tfrac{1}{2}\bfga^a\nabla_b-\tfrac{1}{2}\bfga_b\nabla^a=
\tfrac{1}{2}\delta^a{}_b\bfga^c\nabla_c=\tfrac{1}{2}\delta^a{}_b\Dir_s,
\end{align*}
because any $2$-form is proportional to $\bfom_{ab}$ and the application
of trace results in the constant $\frac{1}{2}$ on the right hand side.
A further simplification gives
${\mathbb Z}^{b}{}_a{}^A{}_B(\delta^a{}_b+\frac{1}{4}\bfga^a\bfga_b)\Dir_s$.

All the three non-trivial contributions quotient from the right by $\Dir_s$,
and it is elementary to see that the overall expression equals to 
the composition ${\mathbb D}^B{}_A \Dir_s$ as claimed in
\eqref{commutdsprolong}. The proof is complete.
\end{proof}

It is an immediate consequence of \eqref{doubleD} that $\mathbb{D}$ preserves 
irreducible components
$\cS_\pm(w)$.  We shall need the explicit formulae for 
${\mathbb D}^B{}_A$ acting on  $\ph \in \cS(-\frac{3}{4})$ and $\ps \in \cS (-\frac{9}{4})$:  
\begin{align}\label{doper34}
& {\mathbb D}^B{}_A \ph = Y^A_aX^b_B(\tfrac{1}{4}\bfga^a\bfga_b 
+\tfrac{1}{2}\delta_b{}^{a})\ph
-\tfrac{1}{2}X^AY_B \ph +X^AX^b_B\nabla_b \ph ,
\\ \label{doper94}
& {\mathbb D}^B{}_A \ps = Y^A_aX^b_B(\tfrac{1}{4}\bfga^a\bfga_b 
+\delta_b{}^{a}) \ps
-\tfrac{3}{2}X^AY_B \ps  + X^AX^b_B\nabla_b \ps.
\end{align}
Combining the BGG splitting $\Si$ in \eqref{split} with the operator $\mathbb{D}$, we obtain
a differential operator $\cO^\si$ on symplectic spinors with the prescribed symbol $\si$. 
Specifically, we have the map
\begin{align} \label{cQ}
\begin{aligned}
\cQ^k: & \, \cE^{(a_1 \ldots a_k)} \to \Diff^k(\cS(w)), \\
&\, \si^{a_1 \ldots a_k} \mapsto 
\bigl( \Si(\si) \bigr)^{A_1 \ldots A_k}{}_{B_1 \ldots B_k}  
\mathbb{D}^{B_1}{}_{A_1} \ldots \mathbb{D}^{B_k}{}_{A_k},
\end{aligned}
\end{align}
where $\Diff^k(\cS(w))$ denotes the space of differential operators acting 
on $\cS(w)$ of order $\leq k$.
It follows from the properties of $\Si$ and $\mathbb{D}$ that $\cQ(\si)$ 
has indeed the symbol $\si$, i.e., 
for $\si^{a_1 \ldots a_k} \in \cE^{(a_1 \ldots a_k)}$ we have 
\begin{equation} \label{cQsi} 
\cQ^k(\si) = \si^{a_1\ldots a_k}\nabla_{a_1}\ldots\nabla_{a_k}+\mathrm{lot}\, : \cS(w) \to \cS(w)
\end{equation}
for any $w \in \mathbb{R}$. Let us note that it follows from invariance
of $\Si$ and $\mathbb{D}$ that $\cQ^k$ yields the projectively invariant bilinear operator
$(\si,\ph) \mapsto \bigl(\cO^k(\si) \bigr) (\ph)$ for $\ph \in \cS(w)$.

Recall that $\Dir_s$ is projectively invariant for a specific weight computed in 
Lemma \ref{sdirinvariance}. To construct higher symmetries of $\Dir_s$, we shall need 
operators $\cQ^k(\si)$ for these specific weights.
We put
\begin{align} \label{Osi}
\begin{aligned}
&\cO^\si := \cQ^k(\si): \cS(-\tfrac34) \to \cS(-\tfrac34), \\
&\ol{\cO}^\si: = \cQ^k(\si): \cS(-\tfrac94) \to \cS(-\tfrac94),
\end{aligned}
\end{align}
where $\cO^\si$ preserves $\cS_\pm(-\frac34)$ and similarly for $\ol{\cO}^\si$.
The crucial information relating $\cO^\si$ to $\Ph(\si)$ 
is provided by the following lemma:

\begin{lemma} \label{naSi}
Let  $\si$ and $\Si(\si)$ be as above. Then
\begin{align*}
&\bigr[\na_b \bigl( \Si(\si) \bigr)\bigr]^{A_1 \ldots A_k}{}_{B_1 \ldots B_k}  
\mathbb{D}^{B_1}{}_{A_1} \ldots \mathbb{D}^{B_k}{}_{A_k} = \\
& = (\tfrac14)^k\bigl( \Ph(\si) \bigr)_b{}^{c_1 \ldots c_k a_1 \ldots a_k}
\bfga_{c_1} \ldots  \bfga_{c_k} \bfga_{a_1} \ldots  \bfga_{a_k}:
\cS(w) \to \cE_b \otimes \cS(w).
\end{align*}
\end{lemma}
\begin{proof}
First, we need to use the properties of 
$\na_b \bigl( \Si(\si) \bigr)$. It follows from the BGG machinery, cf.\ \cite{cd}, that
\begin{align*}
\na_b \bigl( \Si(\si)  \bigr)^{A_1 \ldots A_k}{}_{B_1 \ldots B_k} = &
\mathbb{Z}_{a_1}{}^{c_1}{}^{A_1}{}_{B_1} \ldots \mathbb{Z}_{a_k}{}^{c_k}{}^{A_k}{}_{B_k}
\bigl( \Ph(\si) \bigr)_{bc_1 \ldots c_k}{}^{a_1  \ldots a_k} \\
& +\sum_{i=1}^k \underbrace{\mathbb{X} \dots \mathbb{X}}_i
\underbrace{\mathbb{Z} \ldots \mathbb{Z}}_{k-i} \, \mu_i,
\end{align*}
where we suppressed all abstract indices in the second summand. Here $\mu_i$ is a section of 
suitable bundle which we do need to know explicitly.

Secondly, it follows from \eqref{doubleD} for the composition
\begin{align*}
\mathbb{D}^{B_1}{}_{A_1} & \ldots \mathbb{D}^{B_k}{}_{A_k} = \\ 
&= \mathbb{Z}_{c_1}{}^{a_1}{}^{B_1}{}_{A_1} \ldots \mathbb{Z}_{c_k}{}^{a_k}{}^{B_k}{}_{A_k}
\bigl[ (\tfrac14)^k \tf \bigl( \bfga^{c_1} \bfga_{a_1} \ldots \bfga^{c_k} \bfga_{a_k} \bigl) 
+ \ \text{trace terms} \bigr]  \\
&\ \ \ + \ \text{remaining terms containing $\mathbb{X}$, $\mathbb{Y}$, $\mathbb{Z}$ and $\mathbb{W}$}.
\end{align*}

Now combining $\mathbb{Z} \ldots \mathbb{Z}$-terms of both displays yields the right 
hand side as stated in lemma.
Since $\bigl( \Ph(\si) \bigr)_{bc_1 \ldots c_k}{}^{a_1  \ldots a_k}$ is totally trace free, 
the "trace terms" in the previous display cannot contribute. 
In addition we exploit the notion of homogeneity given by $h(\mathbb{Y}):=1$, $h(\mathbb{Z}):=0$
and $h(\mathbb{X}):=-1$,  extended to concatenations of $\mathbb{Y}$,  $\mathbb{Z}$ and $\mathbb{X}$
in an obvious way. Now since 
$\mathbb{X} \dots \mathbb{X} \mathbb{Z} \ldots \mathbb{Z}$-terms of $\na_b \bigl( \Si(\si)  \bigr)$ 
with at least one $\mathbb{X}$ have homogeneity $\leq -1$ and all ``remaining terms'' in the previous display
have homogeneity $\leq 0$, lemma follows.
\end{proof}
Now we substitute $w = -\frac34$ and use \eqref{commutdsprolong} to conclude:
\begin{theorem} \label{DirOsi}
Assume $(M,[\na])$ is a projectively flat structure, and $\si^{a_1 \ldots a_k} \in \cE^{(a_1 \ldots a_k)}$ for
$k \in{\mN}$. Then the operator $\Dir_s \cO^{\si} - \ol{\cO}^\si \Dir_s: \cS(-\tfrac34) \to \cS(-\tfrac94)$
satisfies
$$
\Dir_s \cO^{\si} - \ol{\cO}^\si \Dir_s = (\tfrac14)^k\bigl( \Ph(\si) \bigr)^{bc_1 \ldots c_k a_1 \ldots a_k}
\bfga_b \bfga_{c_1} \ldots  \bfga_{c_k} \bfga_{a_1} \ldots  \bfga_{a_k}.
$$
In particular, $\cO^{\si}$ is a symmetry of $\Dir_s$ if and only if $\Ph(\si)=0$.
\end{theorem}

Notice that once we define $\cO^{\si}$ and $\ol{\cO}^\si$ as the left multiplication by $\si$
and $\Ph(\si)$ as the differential $d\si$, the theorem holds for $k=0$ as well.

\section{Symbols and construction of higher symmetries of $\Dir_s$}\label{sec:4}

Let us consider a $k$-th order differential operator $\cO$ on weighted symplectic 
spinors. Its principal symbol
is $\si^{a_1 \ldots a_k} \in \cE^{(a_1 \ldots a_k)} \otimes \End(\cS_\pm)$, and 
\begin{equation} \label{EndS}
\End(\cS_\pm) \cong \bigoplus_{i \in \mathbb{N}_0} \cE^{(b_1 \ldots b_{2i})}(-3i)
\end{equation}
since the typical fiber of $\End(\cS)$ is isomorphic to $Cl_s(\mathbb{R}^2,\bfom)$ 
up to a weight, cf.\
\eqref{Cl}. The decomposition \eqref{EndS} implies that $\si^{a_1 \ldots a_k}$  
can be written as
$$
\si^{a_1 \ldots a_k} = \sum_{i \in \mathbb{N}_0} 
\si_i^{a_1 \ldots a_kb_1 \ldots b_{2i}}, 
\quad \si_i^{a_1 \ldots a_kb_1 \ldots b_{2i}} \in \cE^{(a_1 \ldots a_k)(b_1 \ldots b_{2i})}(-3i),
$$
where all $\si_i^{a_1 \ldots a_kb_1 \ldots b_{2i}}$ up to finitely many
vanish. Thus we have for any $w\in\mR$
\begin{align} \label{cO}
\begin{aligned}
\cO: &\, \cS_\pm(w) \to \cS_\pm(w), \\
&\, \ph \mapsto \sum_{i \in \mathbb{N}_0} 
\si_i^{a_1 \ldots a_kb_1 \ldots b_{2i}} \bfga_{b_1} \ldots \bfga_{b_{2i}} \na_{a_1} \ldots \na_{a_k} \ph +\mathrm{lot}
\end{aligned}
\end{align}
where we have explicitly mentioned the leading term of $\cO$.
The subspaces $\cE^{(a_1 \ldots a_k)(b_1 \ldots b_{2i})}$ are in general not irreducible, namely,
we have a distinguished Cartan component 
$\cE^{(a_1 \ldots a_kb_1 \ldots b_{2i})} \subsetneq \cE^{(a_1 \ldots a_k)(b_1 \ldots b_{2i})}$. 
Moreover,  one easily observes that once $\si_i^{a_1 \ldots a_kb_1 \ldots b_{2i}}$ lives in the 
complement to the Cartan component it gives rise to the operator
$$
\si_i^{a_1 \ldots a_kb_1 \ldots b_{2i}} \bfga_{b_1} \ldots \bfga_{b_{2i}} \na_{a_1} \ldots \na_{a_k} 
$$
which factors through the symplectic Dirac operator $\Dir_s = \bfom^{rs} \bfga_r \na_s$ for 
the weight $w = -\tfrac34$.  Henceforth we specialize to this weight and conclude

\begin{lemma} \label{ps}
The operator $\cO$ from \eqref{cO} has, modulo $\Dir_s$,  the principal symbol
$$
\si^{a_1 \ldots a_k} = \sum_{i \in \mathbb{N}_0} 
\si_i^{a_1 \ldots a_kb_1 \ldots b_{2i}}, 
\quad \si_i^{a_1 \ldots a_kb_1 \ldots b_{2i}} \in \cE^{(a_1 \ldots a_kb_1 \ldots b_{2i})}(-3i).
$$
\end{lemma}

We shall turn our attention to properties of the principal symbol for $\cO$ 
in order to become a symmetry 
of $\Dir_s$, i.e.,
$\Dir_s \cO = \cO' \Dir_s$ for some operator $\cO'$. 
The following lemma concerns algebraic properties of such symbols.

\begin{proposition} \label{symalg}
Assume the principal symbol of $\cO$ is non-trivial in the sense that 
$\si_i^{a_1 \ldots a_kb_1 \ldots b_{2i}} \not= 0$ for some $i \geq 1$. 
Then, modulo $\Dir_s$, the composition $\Dir_s \cO$ is a differential 
operator of order $k+1$ and its principal  symbol is non-vanishing.
\end{proposition}

\begin{proof}
We shall consider the composition $\Dir_s \cO$ which is an operator of 
order at most $k+1$. In fact, 
we need just the leading term in this composition. Applying $\bfom^{rs} \bfga_r \na_s$ to $\cO$
as given in \eqref{cO} yields, modulo $\Dir_s$, 
\begin{align*}
\Dir_s \cO &= 
4\sum_{i \in \mathbb{N}_0} i \si_i^{a_1 \ldots a_kb_1 \ldots b_{2i}} 
\bfom^{rs} \bfom_{rb_1} \bfga_{b_2} \ldots \bfga_{b_{2i}} \na_s \na_{a_1} \ldots \na_{a_k} + \mathrm{lot} \\
&= 4\sum_{i \in \mathbb{N}_0} i \si_i^{a_1 \ldots a_kb_1 \ldots b_{2i}} 
\bfga_{b_2} \ldots \bfga_{b_{2i}} \na_{b_1} \na_{a_1} \ldots \na_{a_k} + \mathrm{lot}
\end{align*}
using Lemma \ref{ps} and \eqref{gammaidentity3}. 

Next, consider the principal symbol $\bar{\si}$ of the operator $\ol{\cO} \Dir_s$ for some $\ol{\cO}$. 
Specifically, we shall need the component of  $\bar{\si}$ with $2i-1$ of $\bfga$'s which we denote
by $\bar{\si}_i$. This yields the component in the leading term of $\ol{\cO} \Dir_s$ of the form
\begin{align*}
&\bar{\si}_i^{a_1 \ldots a_kb_1b_2 \ldots b_{2i}} \bfga_{b_2} \ldots \bfga_{b_{2i}} \na_{b_1} \na_{a_1} \ldots \na_{a_k} 
\quad \text{with} \\
& \bar{\si}_i^{a_1 \ldots a_kb_1b_2 \ldots b_{2i}} \in \cE^{(a_1 \ldots a_kb_1) (b_2 \ldots b_{2i})}(-\tfrac{3(2i-1)}{2}).
\end{align*}
We highlight the fact that $\si_i^{a_1 \ldots a_kb_1 \ldots b_{2i}}$ lives in $\cE^{(a_1 \ldots a_kb_1b_2 \ldots b_{2i})}$, whereas
$\bar{\si}_i^{a_1 \ldots a_kb_1b_2 \ldots b_{2i}}$ lives in its complement. Consequently, the differential
operator $\Dir_s \cO - \ol{\cO} \Dir_s$ cannot have order $\leq k$, hence the proof follows.
\end{proof}

\begin{corollary} \label{pssymalg}
If $\cO$ is a symmetry of $\Dir_s$ then, modulo $\Dir_s$, the principal symbol of $\cO$ is
$$
\si^{a_1 \ldots a_k} = \si_0^{a_1 \ldots a_k} \in \cE^{(a_1 \ldots a_k)}.
$$
That is, $\si_i^{a_1 \ldots a_kb_1 \ldots b_{2i}}=0$ for all $i \geq 1$.
\end{corollary}

Given $\si^{a_1 \ldots a_k}$, we can use the operator $\cO^\si$ from \eqref{cO}. 
Assuming $\cO$ is a symmetry of $\Dir_s$, 
both operators $\cO$ and $\cO^\si$ have the same principal symbol, cf.\ \eqref{cQsi}. 
Therefore, we have $\cO = \cO^\si + \cO'$ with $\cO'$ of order $k-1$.  Denote the symbol of $\cO'$ by 
$(\si')^{a_1 \ldots a_{k-1}} \in \cE^{(a_1 \ldots a_{k-1})} \otimes \End(\cS)$. This generally decomposes according
to the number of $\bfga$'s. But since $\cO$ is a symmetry of $\Dir_s$ and $\Dir_s \cO^\si$ has zero 
order by Theorem \ref{DirOsi}, it follows from Proposition \ref{symalg} applied to $\cO'$ that in fact
$(\si')^{a_1 \ldots a_{k-1}} \in \cE^{(a_1 \ldots a_{k-1})}$. Thus we have
$$
\cO = \cO^\si + \cO^{\si'} + \cO''
$$
with $\cO''$ of order $k-2$, and by inductive procedure (with respect to the stritly decreasing order on the algebra of
differential operators) we conclude that after
finitely many steps every symmetry $\cO$ of $\Dir_s$ reduces, modulo $\Dir_s$, into the form
\begin{align} \label{Odec}
\begin{aligned}
& \cO = \cO^{\si_{[k]}} + \cO^{\si_{[k-1]}} + \ldots + \cO^{\si_{[0]}}: 
\cS(-\tfrac34) \to \cS(-\tfrac34), \quad \text{where} \\
& \si_{[i]}^{a_1 \ldots a_i} \in \cE^{(a_1 \ldots a_i)}, \ i=0,\ldots,k.
\end{aligned}
\end{align}
Here we have introduced a new notation: comparing the last two 
displays, we have $\si_{[k]} = \si$, $\si_{[k-1]} = \si'$, etc.

The final statement on the characterization of symbols of higher symmetries 
of $\Dir_s$ is a direct
consequence of Theorem \ref{DirOsi} and Corollary \eqref{Odec}:

\begin{theorem}\label{mainsymbols}
The operator $\cO$ is a symmetry of $\Dir_s$ if and only if
$\Ph(\si_{[i]})=0$ for all $i=0, \ldots, k$, cf.\ \eqref{bgg}.
\end{theorem}

This result, together with Proposition \ref{prol} 
and Theorem \ref{DirOsi}, completes
the description of $\cA^k$ as a vector space, 
$$
\cA^k = \bigoplus_{i=0}^k \boxtimes^k \mathfrak{sl}(3,\mathbb{R}).
$$
The  dimension of the right hand side can be easily computed from 
$\dim \boxtimes^k \mathfrak{sl}(3,\mathbb{R}) = (k+1)^3$, cf. \cite{lm}.

\vspace{1ex}

In order to construct symmetry differential operators for $\Dir_s$ explicitly, we start with 
parallel sections $I^{A}{}_{B}, \bar{I}^{A}{}_{B}\in{\mathcal E}^{A}{}_{B}$ of the adjoint tractor
bundle $\lA_{\mathfrak{d}}$ for $\mathfrak{sl}(3,{\mathbb R})$,
cf. \eqref{covconstsect}. 
Let us first observe that the composition 
$I^{B}{}_{A}\bar{I}^{D}{}_{C}{\mathbb D}^{[A}{}_{[B}{\mathbb D}^{C]}{}_{D]}$
decomposes into irreducibles for $\mathfrak{sl}(3,{\mathbb R})$ according to
\begin{align}\label{ItimesI}
[I,\bar{I}]^{A}{}_{B}= &\, -(I^{A}{}_{R}\bar{I}^{R}{}_{B}-\bar{I}^{A}{}_{R}I^{R}{}_{B}),
\nonumber \\
{\big(I\odot\bar{I}\big)_o}^{A}{}_{B}= &\, \big(I^{A}{}_{R}\bar{I}^{R}{}_{B}
+\bar{I}^{A}{}_{R}I^{R}{}_{B}\big)_o,
\nonumber \\
\langle I,\bar{I}\rangle= &\, I^{R}{}_{S}\bar{I}^{S}{}_{R}.
\end{align}
Note the latter pairing is related to the Killing form $\langle \,,\, \rangle_K$ on $\mathfrak{sl}(3,{\mathbb R})$
by $\langle \,,\, \rangle = \frac16 \langle \,,\, \rangle_K$.
We denote the $\mathfrak{sl}(3,{\mathbb R})$-irreducible Cartan component in the tensor product 
$I^{A}{}_{B}\otimes\bar{I}^{C}{}_{D}$ by
$\big(I\boxtimes\bar{I}\big)^{(A}{}_{(B}{}^{C)}{}_{D)}$. We have
\begin{align}
\big(I\odot\bar{I}\big)^{A}{}_{B}=\big(I^{A}{}_{R}\bar{I}^{R}{}_{B}
+\bar{I}^{A}{}_{R}I^{R}{}_{B}\big)=[I,\bar{I}]^A{}_{B}+
2\bar{I}^{A}{}_{R}I^{R}{}_{B}, 
\end{align}
where the leading term of the adjoint tractor expansion equals
\begin{align}
\bar{I}^{A}{}_{R}I^{R}{}_{B}
= & \,\, {\mathbb Y}_{r}{}^A{}_B\big(v^s(\nabla_s\bar{v}^r)_o 
+\tfrac{1}{6}v^r(\nabla_t\bar{v}^t)-\tfrac{1}{3}v^r(\nabla_t\bar{v}^t)
\big)
\nonumber \\
& + \mathrm{"lower\,\, adjoint\,\, tractor \,\, slots"}
\end{align}
and implies equality in the leading adjoint slot expansion 
\begin{align}
{\big(I\odot\bar{I}\big)_o}^{A}{}_{B} = & \,\,
{\mathbb Y}_{r}{}^A{}_B\big(v^s(\nabla_s\bar{v}^r)
+\bar{v}^s(\nabla_s{v}^r)
-\tfrac{2}{3}v^r(\nabla_t\bar{v}^t)-\tfrac{2}{3}\bar{v}^r(\nabla_t{v}^t)
\big)
\nonumber \\
& + \mathrm{"lower\,\, adjoint\,\, tractor \,\, slots"}.
\end{align}

\begin{lemma}\label{1stordersymm}
The first order differential operator ${\mathcal{S}}^v=I^{A}{}_{B}{\mathbb D}^{B}{}_{A}$,
where $v$ is a solution of \eqref{bgginf} and so is the projective component of the 
parallel tractor $I^{A}{}_{B}$, equals to 
\begin{align}
{\mathcal{S}}^v:=I^{A}{}_{B}{\mathbb D}^{B}{}_{A} = & \, 
v^r\nabla_r+\tfrac{1}{4}(\nabla_rv^s)\bfga^r\bfga_s+\tfrac{1}{3}(\nabla_tv^t).
\end{align}
It is a first order symmetry operator of $\Dir_s$ with symbol $v^r\nabla_r$.
The equivalence relation \eqref{symcond} is for ${\mathcal{S}}^v$ completed by the 
operator
\begin{align}
\bar{\mathcal{S}}^v:=I^{A}{}_{B}\bar{\mathbb D}^{B}{}_{A} = & \, 
v^r\nabla_r+\tfrac{1}{4}(\nabla_rv^s)\bfga^r\bfga_s+\tfrac{5}{6}(\nabla_tv^t).
\end{align}
The Lie algebra structure on first order symmetry operators
given by parallel tractors $I{}^{A}{}_{B}, \tilde{I}{}^{A}{}_{B}\in\cE^A{}_B$
is
$[I,\tilde{I}]^{A}{}_{B}[{\mathbb D},{\mathbb D}]^{B}{}_{A}$ with  the leading term
$[v,\tilde{v}]^a\nabla_a$.  Here $[v,\tilde{v}]^a$ is the Lie bracket of vector
fields $v$ and $\tilde{v}$.
\end{lemma}

\begin{proof} The first part of the claim follows from Theorem \ref{mainsymbols},
and the explicit formulas for ${\mathcal{S}}^v$ and $\bar{\mathcal{S}}^v$ from  
\eqref{splittingsl3adjoint}, \eqref{doper34}, \eqref{doper94}. Analogously, we have
\begin{align}
[I,\tilde{I} & ]^{A}{}_{B} = \,\, -
{\mathbb Y}_{a}{}^A{}_{R}v^a\big( {\mathbb Z}_{r}{}^s{}_{B}{}^{R}
((\nabla_s\tilde{v}^r)_o+\tfrac{1}{6}\delta_s{}^r(\nabla_t\tilde{v}^t)\big) 
-\tfrac{1}{3}{\mathbb W}^{R}{}_{B}(\nabla_t\tilde{v}^t)
\big)
\nonumber \\
& \,\, -
{\mathbb Y}_{a}{}^R{}_{B}\tilde{v}^a\big( {\mathbb Z}_{r}{}^s{}_{R}{}^{A}
((\nabla_s{v}^r)_o+\tfrac{1}{6}\delta_s{}^r(\nabla_t{v}^t)\big) 
-\tfrac{1}{3}{\mathbb W}^{A}{}_{R}(\nabla_t{v}^t)
\big)
\nonumber \\
& \,\, +
{\mathbb Y}_{a}{}^R{}_{B}v^a\big( {\mathbb Z}_{r}{}^s{}_{R}{}^{A}
((\nabla_s\tilde{v}^r)_o+\tfrac{1}{6}\delta_s{}^r(\nabla_t\tilde{v}^t)\big) 
-\tfrac{1}{3}{\mathbb W}^{A}{}_{R}(\nabla_t\tilde{v}^t)
\big)
\nonumber \\
& \,\, +
{\mathbb Y}_{a}{}^A{}_{R}\tilde{v}^a\big( {\mathbb Z}_{r}{}^s{}_{B}{}^{R}
((\nabla_s{v}^r)_o+\tfrac{1}{6}\delta_s{}^r(\nabla_t{v}^t)\big) 
-\tfrac{1}{3}{\mathbb W}^{R}{}_{B}(\nabla_t{v}^t)
\big)\nonumber \\
& + \mathrm{"lower\,\, adjoint\,\, tractor\,\, slots"}
\nonumber \\
= & \,\, {\mathbb Y}_{a}{}^A{}_{B}\big( 
v^s(\nabla_s\tilde{v}^a)-\tilde{v}^s(\nabla_s{v}^a)
\big) + \mathrm{"lower\,\, adjoint\,\, tractor\,\, slots"}
\nonumber \\
= & \,\, {\mathbb Y}_{a\,\,\,\, B}^{\, A}[v,\tilde{v}]^a
 + \mathrm{"lower\,\, adjoint\,\, tractor\,\, slots"},
\end{align}
which proves the second half of the claim.
\end{proof}

\section{Algebra of higher symmetries of $\Dir_s$}\label{sec:5}

In the present section we analyze the algebra structure on the vector space of higher 
symmetries of $\Dir_s$. Following Lemma \ref{1stordersymm}, the 
first order symmetries generate the tensor algebra of higher symmetry 
differential operators and hence we have to produce
$\mathfrak{sl}(3,{\mathbb R})$-invariant decomposition of the composition
of operators ${\mathbb D}^{B}{}_{A}$. As we shall argue in the end of this section, it is sufficient to compute
it just for ${\mathbb D}^{B}{}_A{\mathbb D}^{D}{}_C$.
Throughout this section we assume $w=-\tfrac34$. 

\begin{lemma}\label{ddcomposition}
The composition of two operators  ${\mathbb D}^{B}{}_A$, 
see \eqref{doper34}, equals to
\begin{align}
{\mathbb D }^{B} & {}_A \, {\mathbb D}^{D}{}_C = 
\nonumber \\
&=  {\mathbb Z}_b{}^a{}^B{}_A{\mathbb Z}_d{}^c{}^D{}_C
\big(\tfrac{1}{16}\bfga^b\bfga_a\bfga^d\bfga_c+\tfrac{1}{8}
(\bfga^b\bfga_a\delta^d_c+\bfga^d\bfga_c\delta^b_a)
+\tfrac{1}{4}\delta_a^b\delta_c^d\big)
\nonumber \\
& -\tfrac{1}{2}({\mathbb Z}_b{}^a{}^B{}_A{\mathbb W}^D{}_C
+{\mathbb Z}_b{}^a{}^D{}_C{\mathbb W}^B{}_A)
(\tfrac{1}{4}\bfga^b\bfga_a+\tfrac{1}{2}\delta^b_a)
 +\tfrac{1}{4}{\mathbb W}^B{}_A{\mathbb W}^D{}_C
\nonumber \\
& - {\mathbb X}^b{}^B{}_A{\mathbb Y}_d{}^D{}_C(\tfrac{1}{4}\bfga^d\bfga_b+\delta^d_b)
 + {\mathbb X}^b{}^B{}_A{\mathbb Z}_d{}^c{}^D{}_C
(\tfrac{1}{4}\bfga^d\bfga_c\nabla_b+\tfrac{1}{2}\delta^d_c\nabla_b+\delta^d_b\nabla_c)
\nonumber \\
& +{\mathbb Z}_b{}^a{}^B{}_A{\mathbb X}^d{}^D{}_C
(\tfrac{1}{4}\bfga^b\bfga_a+\tfrac{1}{2}\delta^b_a)\nabla_d
 -\tfrac{3}{2}{\mathbb X}^b{}^B{}_A{\mathbb W}^D{}_C\nabla_b
-\tfrac{1}{2}{\mathbb W}^B{}_A{\mathbb X}^d{}^D{}_C\nabla_d
\nonumber \\
& + {\mathbb X}^b{}^B{}_A{\mathbb X}^d{}^D{}_C
(\nabla_b\nabla_d-\tfrac{1}{4}P_{br}\bfga^r\bfga_d-P_{bd}). \label{compositiondd}
\end{align}
\end{lemma} 

\begin{proof}
The proof is a straightforward computation based on
\eqref{doper34}, \eqref{covderinj} and \eqref{injoperpartscalprod}.
\end{proof}

The composition ${\mathbb D}^{B}{}_A{\mathbb D}^{D}{}_C$ can be invariantly decomposed 
according to the $\mathfrak{sl}(3,\mR)$-module structure on the second tensor
power of its adjoint representation. We shall first examine the 
skew-symmetric part of this tensor product,
which induces the Lie algebra structure on the linear
span of symmetry operators ${\mathbb D}^{A}{}_{B}$. In fact,
we shall show that the skew-symmetric component in their composition contains
just the adjoint representation. 

\begin{lemma}
The skew-symmetric component of \eqref{compositiondd} is given by
\begin{align}\label{skewsymmetriccomp}
& {\mathbb D}^{B}{}_A{\mathbb D}^{D}{}_C 
- \,{\mathbb D}^{D}{}_{C}{\mathbb D}^{B}{}_{A} = 
\nonumber \\
&\quad= \tfrac{1}{8} {\mathbb Z}_b{}^a{}^B{}_A{\mathbb Z}_d{}^c{}^D{}_C
(\bfga^b\bfga_c\delta_a{}^d-\bfga^d\bfga_a\delta_c{}^b
+\bfga_c\bfga_a\bfom^{bd}+\bfga^b\bfga^d\bfom_{ac})
\nonumber \\
& \quad - ({\mathbb X}^b{}^B{}_A{\mathbb Y}_d{}^D{}_C
-{\mathbb Y}_d{}^B{}_A{\mathbb X}^b{}^D{}_C)
(\tfrac{1}{4}\bfga^d\bfga_b+\delta^d_b)
\nonumber \\
& \quad + ({\mathbb X}^b{}^B{}_A{\mathbb Z}_d{}^c{}^D{}_C
-{\mathbb Z}_d{}^c{}^B{}_A{\mathbb X}^b{}^D{}_C)
\delta_b{}^d\nabla_c
\nonumber \\
& \quad -({\mathbb X}^b{}^B{}_A{\mathbb W}^D{}_C
-{\mathbb W}^B{}_A{\mathbb X}^b{}^D{}_C)\nabla_b \, .
\end{align}
\end{lemma}

\begin{proof}
This is a straightforward consequence of Lemma \ref{ddcomposition}. We notice that 
the contribution 
\begin{align}
{\mathbb X}^b{}^B{}_A{\mathbb X}^d{}^D{}_C
\big(-\tfrac{1}{4}R_{bd}{}^r{}_s\bfga_r\bfga^s-\tfrac{1}{2}P_{r[b}\bfga^r\bfga_{d]}\big)
\end{align}
is trivial due to the projective curvature tensor identity in real dimension $2$, 
$R_{bd}{}^r{}_s\bfga_r\bfga^s=2P_{s[d}\bfga_{b]}\bfga^s$, 
and an elementary identity in the symplectic Clifford algebra.
\end{proof}
\begin{lemma}\label{asymmcompos}
The commutator in the composition of
${\mathbb D}^{B}{}_{A}$ and ${\mathbb D}^{D}{}_{C}$,
\begin{align}
[{\mathbb D},{\mathbb D}]_A{}^{B}=(\mathrm{tr}({\mathbb D}\wedge{\mathbb D}))_A{}^{B}
={\mathbb D}^{B}{}_{R}{\mathbb D}^{R}{}_{A} 
- \,{\mathbb D}^{R}{}_{A}{\mathbb D}^{B}{}_{R},
\end{align}
is related to \eqref{skewsymmetriccomp} via
\begin{align}
{\mathbb D}^{B}{}_{A}{\mathbb D}^{D}{}_{C} 
- \,{\mathbb D}^{D}{}_{C}{\mathbb D}^{B}{}_{A}
=\tfrac{2}{3}\delta_{[A}{}^{(D}[{\mathbb D},{\mathbb D}]_{C]}{}^{B)}
-\tfrac{2}{3}\delta_{(C}{}^{[B}[{\mathbb D},{\mathbb D}]_{A)}{}^{D]} .
\end{align}
\end{lemma}

\begin{proof}
We recall the formula for the tractor volume form 
${\bfep}^{[ABC]}:=X^{[A}Y_b{}^{B}Y^{C]}{}_c\bfom^{bc}$, which satisfies
the following properties:
\begin{align}\label{omegaidentities}
&{\bfep}^{EAC}X^a{}_AX^c{}_C = \,\tfrac{1}{3}X^E\bfom^{ac},
\nonumber \\
&{\bfep}^{EAC}X^b{}_{[A}Y_{C]} = -\tfrac{2}{3}X^AY_e{}^EY_c{}^C\bfom^{ec}X^b{}_{[A}Y_{C]}
=\tfrac{1}{3}Y_e{}^E\bfom^{eb}.
\end{align}
The contraction of 
\begin{align}
2{\mathbb D}^{(B}{}_{[A}{\mathbb D}^{D)}{}_{C]} = & \,\,\, 
\tfrac{1}{2}Y^B{}_bY^D{}_dX^a{}_AX^c{}_C
\big( \bfga^{(b}\bfga^{d)}\bfom_{ac}+2\delta^{(b}{}_{[a}\bfga^{d)}\bfga_{c]}\big)
\nonumber \\
& - X^{(B}Y^{D)}{}_dX^b{}_{[A}Y_{C]}\big( \tfrac{1}{2}\bfga^{d}\bfga_{b}+
\delta^{d}{}_{b}\big)
\nonumber \\
& +2X^{(B}Y^{D)}{}_dX^b{}_{[A}X^c{}_{C]}\delta^{d}{}_{[b}\nabla_{c]}
\nonumber \\
& -2X^{(B}X^{D)}X^b{}_{[A}Y_{C]}\nabla_b
\end{align}
by ${\bfep}^{EAC}$ yields (notice that the last two terms on the last display cancel out
after contraction by ${\bfep}$ due to \eqref{omegaidentities})
\begin{align}
{\bfep}^{EAC}{\mathbb D}^{(B}{}_{[A}{\mathbb D}^{D)}{}_{C]} =
\tfrac{1}{3}X^{(E}Y_b{}^{B}Y_d{}^{D)}\bfga^{(b}\bfga^{d)}
-\tfrac{1}{3}Y_e{}^{(E}X^{B}Y_d{}^{D)}\bfga^{(d}\bfga^{b)}=0.
\end{align}
This result implies immediately that the only irreducible component 
in the skew-symmetric part of the composition is the Lie bracket, 
which completes the proof. 
\end{proof}

\begin{lemma}
The skew-symmetric component in the composition 
of two operators ${\mathbb D}^A{}_B$ contains 
only the trace part, and can consequently be 
simplified as
\begin{align}\label{commd3d}
[{\mathbb D},{\mathbb D}]^B{}_{C}
={\mathbb D}^{B}{}_{R}{\mathbb D}^{R}{}_{C} 
- \,{\mathbb D}^{R}{}_{C}{\mathbb D}^{B}{}_{R}
=3{\mathbb D}^{B}{}_{C}.
\end{align}
\end{lemma}

\begin{proof}
We apply the trace $\delta^A_D$ to \eqref{skewsymmetriccomp}, and 
use the identities \eqref{covderinj}, \eqref{injoperpartscalprod} and 
\eqref{injoperscalprod}:
\begin{align}
[{\mathbb D},{\mathbb D}]^B{}_{C} = & \, 
\tfrac{1}{8}{\mathbb Z}_b{}^{c}{}^{B}{}_{C}
\big( 2\bfga^b\bfga_c +2\delta_c{}^b+\bfga_c\bfga^b+\bfga^b\bfga_c\big)
\nonumber \\
& -\delta^b{}_d{\mathbb W}^{B}{}_{C}(\tfrac{1}{4}\bfga^d\bfga_b +\delta^d{}_b) 
+{\mathbb Z}_d{}^{b}{}^{B}{}_{C}(\tfrac{1}{4}\bfga^d\bfga_b +\delta^d{}_b)
\nonumber \\
& + {\mathbb X}^c{}^{B}{}_{C}\delta^b{}_d\delta_b{}^d\nabla_c
+ {\mathbb X}^b{}^{B}{}_{C}\nabla_b .
\end{align}
By elementary manipulations in the symplectic Clifford algebra, the last 
display equals to 
\begin{align}
3\big(
{\mathbb Z}_b{}^{c}{}^{B}{}_{C}(\tfrac{1}{4}\bfga^b\bfga_c 
+\tfrac{1}{2}\delta_c{}^{b})
-\tfrac{1}{2}{\mathbb W}^{B}{}_{C}
+{\mathbb X}^b{}^{B}{}_{C}\nabla_b
\big),
\end{align}
which proves by \eqref{doper34} the statement of the lemma.
\end{proof}

Now we pass to the analysis of the symmetric part 
\begin{align}
({\mathbb D}\odot{\mathbb D})^{B}{}_{A}{}^{D}{}_{C} =
{\mathbb D}^{B}{}_{A}{\mathbb D}^{D}{}_{C} 
+ \,{\mathbb D}^{D}{}_{C}{\mathbb D}^{B}{}_{A}
\end{align}
of the composition
${\mathbb D}^{A}{}_{B}{\mathbb D}^{C}{}_{D}$. 
The following notation turns out to be convenient for our purposes:
\begin{align} 
\big({\mathbb D}^2\big)^{A}{}_{B} = \,
{\mathbb D}^{A}{}_{R}{\mathbb D}^{R}{}_{B} 
+ \,{\mathbb D}^{R}{}_{B}{\mathbb D}^{A}{}_{R},
\, \langle {\mathbb D}, {\mathbb D} \rangle = \,
{\mathbb D}^{S}{}_{R}{\mathbb D}^{R}{}_{S},
\, {\mathbb D}^2_o = \, \mathrm{tf}\big({\mathbb D}^2\big).
\end{align}

\begin{lemma}\label{symmpartcomp}
The composition ${\mathbb D}^{R}{}_{B}{\mathbb D}^{A}{}_{R}$
of ${\mathbb D}^{R}{}_{B}$  is given by 
\begin{align}\label{ddsquare}
{\mathbb D}^{R}{}_{B}{\mathbb D}^{A}{}_{R} = \, &
\tfrac{1}{8} {\mathbb Z}_d{}^{a}{}^{A}{}_{B} (-3\bfga^d\bfga_a -10\delta_a{}^d)
+ {\mathbb X}^a{}^{A}{}_{B}\delta_b{}^d(\tfrac{1}{4}\bfga^b\bfga_a +\tfrac{1}{2}\delta_a{}^b)\nabla_d
\nonumber \\
& - \tfrac{3}{2}{\mathbb X}^c{}^{A}{}_{B}\nabla_c +\tfrac{1}{4}{\mathbb W}^{A}{}_{B}.
\end{align} 
Consequently, we have
\begin{align}
& {({\mathbb D}^2)}^A{}_{B}  = -\delta^{A}{}_{B}
+\tfrac{1}{2}{\mathbb X}^{b}{}^A{}_{B}\bfga_b\Dir_s ,
\nonumber \\
& {({\mathbb D}_o^2)}^A{}_{B}\, = 
\tfrac{1}{2}{\mathbb X}^b{}^{A}{}_{B}\bfga_b\Dir_s. 
\label{dsquare}
\end{align}
\end{lemma}

\begin{proof}
The first equality \eqref{ddsquare} is the trace component of  
Lemma \eqref{ddcomposition}. 
By \eqref{commd3d}, we have 
\begin{align}
{({\mathbb D}^2)}^A{}_{B} = &\,\, {\mathbb D}^{A}{}_{R}{\mathbb D}^{R}{}_{B} 
+ {\mathbb D}^{R}{}_{B}{\mathbb D}^{A}{}_{R}=
({\mathbb D}^{A}{}_{R}{\mathbb D}^{R}{}_{B} -
 {\mathbb D}^{R}{}_{B}{\mathbb D}^{A}{}_{R})
+2\,{\mathbb D}^{R}{}_{B}{\mathbb D}^{A}{}_{R}
\nonumber \\
=&\,\, 3{\mathbb D}^{A}{}_{B}+2\,{\mathbb D}^{R}{}_{B}{\mathbb D}^{A}{}_{R},
\end{align}
and the substitution from an elementary formula 
\begin{align}\label{dsquaredecomposition}
{\mathbb D}^{R}{}_{B}{\mathbb D}^{A}{}_{R}=
-\tfrac{3}{2}{\mathbb D}^A{}_{B}
-\tfrac{1}{2}\delta^{A}{}_{B}
+\tfrac{1}{4}{\mathbb X}^b{}^{A}{}_{B}\bfga_b\Dir_s 
\end{align}
implies \eqref{dsquare}. We used the identity decomposition 
$\delta^{A}{}_{B}={\mathbb Z}_a{}^{a}{}^{A}{}_{B}+{\mathbb W}^{A}{}_{B}$
in order to separate the trace and the trace free components.
The proof is complete.
\end{proof}

\begin{lemma}\label{symmcompos}
The symmetric component in the composition of two 
operators ${\mathbb D}^{A}{}_{B}$ equals to
\begin{align}
\tfrac{1}{2}({\mathbb D}^{A}{}_{B}{\mathbb D}^{C}{}_{D} 
+ \,{\mathbb D}^{C}{}_{D}{\mathbb D}{}^{A}{}_{B})
={\mathbb D}^{(A}{}_{(B}{\mathbb D}^{C)}{}_{D)}+
{\mathbb D}^{[A}{}_{[B}{\mathbb D}^{C]}{}_{D]},
\end{align}
where
\begin{align}
{\mathbb D}^{(A}{}_{(B}{\mathbb D}^{C)}{}_{D)}=&\,\,
\mathrm{tf}\big( {\mathbb D}^{(A}{}_{(B}{\mathbb D}^{C)}{}_{D)} \big)
+\tfrac{2}{5}\delta^{(A}{}_{(B}{\big({\mathbb D}_o^2\big)}^{C)}{}_{D)}
\nonumber \\
& +\tfrac{1}{12}\delta^{(A}{}_{(B}\delta^{C)}{}_{D)}\langle {\mathbb D}, {\mathbb D} \rangle,
\nonumber\\\label{squareab}
{\mathbb D}^{[A}{}_{[B}{\mathbb D}^{C]}{}_{D]}=&\,
-2\delta^{[A}{}_{[B}{\big({\mathbb D}_o^2\big)}^{C]}{}_{D]}
-\tfrac{1}{6}\delta^{[A}{}_{[B}\delta^{C]}{}_{D]}\langle {\mathbb D}, {\mathbb D} \rangle .
\end{align}
Moreover, 
$\langle {\mathbb D}, {\mathbb D} \rangle =-\frac{3}{2}$.
\end{lemma}

\begin{proof}
As for the term ${\mathbb D}^{(A}{}_{(B}{\mathbb D}^{C)}{}_{D)}$, we can write
\begin{align}
{\mathbb D}^{(A}{}_{(B}{\mathbb D}^{C)}{}_{D)}=&\, 
\mathrm{tf}\big( {\mathbb D}^{(A}{}_{(B}{\mathbb D}^{C)}{}_{D)} \big)
+\tilde{A}\delta^{(A}{}_{(B}{\big({\mathbb D}_o^2\big)}^{C)}{}_{D)}
\nonumber \\
& +\tilde{B}\delta^{(A}{}_{(B}\delta^{C)}{}_{D)}\langle {\mathbb D}, {\mathbb D} \rangle
\end{align}
for some $\tilde{A}, \tilde{B}\in{\mathbb C}$. The double trace 
$\delta_C{}^{D}, \delta_A{}^{B}$
of the first equality in \eqref{squareab} implies $\tilde{B}=\frac{1}{12}$, while its first 
trace $\delta_C{}^{D}$
gives after the substitution for $\tilde{B}$ the value $\tilde{A}=\frac{2}{5}$. The second equality 
in \eqref{squareab} for   
${\mathbb D}^{[A}{}_{[B}{\mathbb D}^{C]}{}_{D]}$ is proved analogously.

The last claim follows by taking the trace component of 
the first equality in \eqref{dsquare}, or equivalently \eqref{dsquaredecomposition}.
The proof is complete.
\end{proof}

\begin{lemma}\label{preparationlemma}
Let $I^{A}{}_{B},\bar{I}^{A}{}_{B}\in {\mathcal E}^{A}{}_{B}$ be parallel sections 
of the projective adjoint tractor bundle. Then 
$I^{A}{}_{B}{\mathbb D}^{B}{}_{A}, \bar{I}^{A}{}_{B}{\mathbb D}^{B}{}_{A}$
are first order symmetry operators of $\Dir_s$, and there is 
$\mathfrak{sl}(3,{\mathbb R})$-invariant decomposition
\begin{align}
I^{A}{}_{B}{\mathbb D}^{B}{}_{A} \bar{I}^{C}{}_{D}{\mathbb D}^{D}{}_{C}= & \,\,
\big(I\boxtimes\bar{I}\big)^{(A}{}_{(B}{}^{C)}{}_{D)}{\mathbb D}^{B}{}_{A}{\mathbb D}^{D}{}_{C}
+\tfrac{6}{5}\big(I\odot\bar{I}\big)^{A}{}_{B}{\big({\mathbb D}_o^2\big)}^{B}{}_{A}
\nonumber \\
& +\tfrac{1}{6}[I,\bar{I}]^{A}{}_{B}[{\mathbb D},{\mathbb D}]^B{}_{A}
+\tfrac{1}{8}\langle I,\bar{I}\rangle\langle {\mathbb D}, {\mathbb D} \rangle.
\end{align}
\end{lemma}

\begin{proof}
By \eqref{commutdsprolong} and the assumption that $I,\bar{I}$
are parallel tractors, $I^{A}{}_{B}{\mathbb D}^{B}{}_{A}$ and 
$\bar{I}^{A}{}_{B}{\mathbb D}^{B}{}_{A}$
are first order symmetry operators of $\Dir_s$.

Secondly, we have by Lemma \ref{asymmcompos} and Lemma \ref{symmcompos}
\begin{align}
I^{A}{}_{B}\bar{I}^{C}{}_{D} & {\mathbb D}^{B}{}_{A}{\mathbb D}^{D}{}_{C}= \,
\tfrac{1}{2}I^{A}{}_{B}\bar{I}^{C}{}_{D}({\mathbb D}^{B}{}_{A}{\mathbb D}^{D}{}_{C} 
+ \,{\mathbb D}^{D}{}_{C}{\mathbb D}^{B}{}_{A})
\nonumber \\
& +\tfrac{1}{2}I^{A}{}_{B}\bar{I}^{C}{}_{D}({\mathbb D}^{B}{}_{A}{\mathbb D}^{D}{}_{C} 
- \,{\mathbb D}^{D}{}_{C}{\mathbb D}^{B}{}_{A})
\nonumber \\
= & \, \big(I\boxtimes\bar{I}\big)^{(A}{}_{(B}{}^{C)}{}_{D)}{\mathbb D}^{B}{}_{A}{\mathbb D}^{D}{}_{C}
\nonumber \\
& +I^{A}{}_{B}\bar{I}^{C}{}_{D}\big(
\tfrac{2}{5}\delta^{(B}{}_{(A}{\big({\mathbb D}_o^2\big)}^{D)}{}_{C)}
-2\delta^{[B}{}_{[A}{\big({\mathbb D}_o^2\big)}^{D]}{}_{C]}\big)
\nonumber \\
& +I^{A}{}_{B}\bar{I}^{C}{}_{D}\big(
\tfrac{1}{12}\delta^{(B}{}_{(A}\delta^{D)}{}_{C)}\langle {\mathbb D}, {\mathbb D} \rangle
-\tfrac{1}{6}\delta^{[B}{}_{[A}\delta^{D]}{}_{C]}\langle {\mathbb D}, {\mathbb D} \rangle
\big)
\nonumber \\ 
& +\tfrac{1}{3}I^{A}{}_{B}\bar{I}^{C}{}_{D}
\big(\delta_{[A}{}^{(B}[{\mathbb D},{\mathbb D}]_{C]}{}^{D)}
-\delta_{(A}{}^{[B}[{\mathbb D},{\mathbb D}]_{C)}{}^{D]}\big),  \nonumber
\end{align}
and the substitution of identities \eqref{ItimesI} yields the claim.
The proof is complete.
\end{proof}

\begin{theorem}\label{maintheorem}
Let $I^{A}{}_{B},\bar{I}^{A}{}_{B}\in {\mathcal E}^{A}{}_{B}$ be parallel sections 
of the projective adjoint tractor bundle $\lA_{\mathfrak{d}}$
corresponding to the first order symmetry operators 
$S^v=I^{A}{}_{B}{\mathbb D}^{B}{}_{A}$ and 
$S^{\bar{v}}=\bar{I}^{A}{}_{B}{\mathbb D}^{B}{}_{A}$
of $\Dir_s$. Then their composition equals
\begin{align}
S^v\circ & S^{\bar{v}}= \, I^{A}{}_{B}{\mathbb D}^{B}{}_{A} \bar{I}^{C}{}_{D}{\mathbb D}^{D}{}_{C}
\nonumber \\
 &=\,\big(I\boxtimes\bar{I}\big)^{(A}{}_{(B}{}^{C)}{}_{D)}
{\mathbb D}^{B}{}_{A}{\mathbb D}^{D}{}_{C}
+\tfrac12[I,\bar{I}]^{A}{}_{B}{\mathbb D}^{B}{}_{A}
-\tfrac{3}{16}\langle I,\bar{I}\rangle\quad \mathrm{mod}\quad \Dir_s ,
\end{align}
hence the symmetry algebra of $\Dir_s$ is isomorphic to the quotient of the tensor algebra
$\bigoplus\limits_{k=0}^\infty\otimes^k\big(\mathfrak{sl}(3,{\mathbb R})\big)$
by a two sided ideal generated by quadratic relations
\begin{align}\label{idealintensor}
I\otimes\bar{I}-I\boxtimes\bar{I}-\tfrac{1}{2}[I,\bar{I}]+\tfrac{3}{16}\langle I,\bar{I}\rangle .
\end{align}
Equivalently, the symmetry algebra of $\Dir_s$ is the quotient of the universal enveloping 
algebra $U(\mathfrak{sl}(3,{\mathbb R}))$ by a two sided ideal generated by quadratic 
relations 
\begin{align}
I\bar{I}+\bar{I}I-2I\boxtimes\bar{I}+\tfrac38\langle I,\bar{I}\rangle,\quad 
I,\bar{I}\in \mathfrak{sl}(3,{\mathbb R}) .
\end{align}
\end{theorem}

\begin{proof}
The proof goes along the same lines as in e.g., \cite{eas}, so we shall give just 
a brief account of its exposition. 

The identification of $\gog=\mathfrak{sl}(3,{\mathbb R})$ with differential symmetries
is given by the mapping $I^{A}{}_{B}\mapsto I^{A}{}_{B}{\mathbb D}^{B}{}_{A}$,
where $I^{A}{}_{B}$ is a parallel section of the projective adjoint tractor bundle 
$\lA_{\mathfrak{d}}$. This extends to
\begin{align}\label{tenscompos}
\gog\otimes\ldots\otimes\gog\mapsto (I^{A}{}_{B}{\mathbb D}^{B}{}_{A}) 
\ldots (\bar{I}^{C}{}_{D}{\mathbb D}^{D}{}_{C})
\end{align}
with $I^{A}{}_{B}\otimes\ldots\otimes \bar{I}^{D}{}_{C}\in \gog\otimes\ldots\otimes\gog$,
and hence to the full tensor algebra $\bigoplus_k\otimes^k\gog$ by linearity.

The first step in the proof is to express the composition 
$I^{A}{}_{B}{\mathbb D}^{B}{}_{A}\bar{I}^{C}{}_{D}{\mathbb D}^{D}{}_{C}$
for $I^{A}{}_{B}, \bar{I}^{A}{}_{B}\in\gog$ in terms of canonical 
symmetries. This was already done in Lemma \ref{preparationlemma}. 

To finish the proof, we have the following observations. The mapping
\eqref{tenscompos} determines an associative algebra morphism
$\bigoplus_k\otimes^k\gog\to\cA$ with $\cA$ the algebra of symmetries, cf.\  the paragraph beyond 
\eqref{symcond}, which is surjective as a consequence
of the fact that the canonical symmetries $I^{A}{}_{B}$ arise in the range 
of \eqref{tenscompos} (cf., Lemma \ref{preparationlemma}). We want to find
all relations, that is to identify the two-sided ideal of this algebra
morphism. As we already identified the generators of the ideal \eqref{idealintensor},
it remains to show that this ideal is large enough to have $\cA$
as the resulting quotient. 

Since we know $\cA$ as a vector space, 
cf. Section \ref{sec:4}, it is sufficient to consider the associated 
graded algebra (the symbol algebra of $\cA$.) The corresponding graded 
ideal contains $I\otimes \bar{I}-I\boxtimes \bar{I}$ for $I, \bar{I}\in\gog$,
hence contains the skew-symmetric component of the tensor product $\gog\wedge\gog$. 
Therefore, we can pass to symmetric tensors $\odot\gog$ in the tensor algebra
and write ${\fam2 I}$ for the ideal in $\odot\gog$ defined as the image of 
\eqref{idealintensor}. Now we claim that as for the associated graded
$\cA=\bigoplus_k\cA^k$, where the $\cA^k$ are defined as the $\gog$-submodules
satisfying $\cA^k=\{F^{(A_1}{}_{(B_1}{}^{\ldots}{}_{\ldots}{}^{A_k)}{}_{B_k)}
\quad \text{with all traces zero}\}\subset \odot^k\gog$. 

To finish the proof, we need
to show the vector space decomposition $\odot^k\gog=\cA^k\oplus{\fam2 I}^k$
for ${\fam2 I}^k:={\fam2 I}\cap \odot^k\gog$. This is based on the following 
observation:
\begin{align}
\big(\gog\otimes\cA^{k-1}\big)\cap\big(\cA^{k-1}\otimes\gog\big)=\cA^k,
\quad  
\end{align}
which is elementary to check directly for $\mathfrak{sl}(3,{\mathbb R})$
(and at the same time holds for $\mathfrak{sl}(n,{\mathbb R})$ in general): the inclusion 
$\supseteq$ is obvious, and to prove the inclusion $\subseteq$
we consider $F^{A_1}{}_{B_1}{}^{\ldots}{}_{\ldots}{}^{A_{k}}{}_{B_{k}}$ living in the intersection
on the left hand side of the display. Then
$$
F^{A_1}{}_{B_1}{}^{\ldots}{}_{\ldots} {}^{A_i} {}^{\ldots} {}^{A_j} {}^{\ldots}{}_{\ldots} 
{}^{A_{k}}{}_{B_{k}}
= F^{A_1}{}_{B_1}{}^{\ldots}{}_{\ldots} {}^{A_j} {}^{\ldots} {}^{A_i} {}^{\ldots}{}_{\ldots} 
{}^{A_{k}}{}_{B_{k}}
$$
for any $1 \leq i < j \leq k$. A similar conclusion holds for  the symmetry in the collection
of lower indices as well as for the trace-freeness.


The final step relies on the following standard fact in the representation theory of 
simple Lie algebras. There is a projection $\odot^k\gog\to\cA^k$ such that the induced
projections $P^k: \odot^k\gog\to \gog\otimes\cA^{k-1}$ and
$\tilde{P}^k: \odot^k\gog\to \cA^{k-1}\otimes\gog$ have their kernels
in $\gog\otimes{\fam2 I}^{k-1}$ and ${\fam2 I}^{k-1}\otimes\gog$,
respectively. In particular, it is contained in ${\fam2 I}^k$ for both cases. By 
standard dimensional considerations in linear algebra, 
\begin{align}
\odot^k\gog=\big(\mathrm{Im}(P^k)\cap \mathrm{Im}(\tilde{P}^k)\big)\oplus
\big(\mathrm{Ker}(P^k)+ \mathrm{Ker}(\tilde{P}^k)\big)
\end{align}
for all $k\geq 3$, so the claim above follows. This completes the proof of theorem.
\end{proof}

It is well known (cf., \cite{o}, Section $4$) that the representation of 
$\mathfrak{sl}(3,{\mathbb R})$ on $\mathrm{Ker}(\Dir_s)$ is an unitary 
irreducible representation
equivalent to the exceptional representation associated
with the minimal nilpotent orbit of $\mathfrak{sl}(3,{\mathbb R})$, 
cf. \cite{dix}, \cite{jos}, \cite{tor}.
This result is based on the analysis of K-types in the underlying 
Harish-Chandra module of $\mathrm{Ker}(\Dir_s)$.
We note that in the case of simple Lie algebras $A_n, n\in{\mathbb N}$, the 
Joseph ideal in $U(\gog)$ is not uniquely defined and 
there is a one parameter family of completely prime primitive ideals having
as its associated variety the minimal nilpotent orbit.  


\section{Comments and open questions}

Let us conclude by observing that higher symmetries of $\Dir_s$ for dimensions $2n>2$
are not induced from a semi-simple Lie algebra of symmetries. In particular, it is 
straightforward to see that for $({\mR}^{2n},\omega)$ and the flat symplectic connection $\na$,
a general first order symmetry differential  operator is of the form 
\begin{align*} 
{\fam2 O} = v^a\nabla_a+
\sum\limits_{j=0}^\infty\gamma^{a_1}\ldots\gamma^{a_{2j}}w^j_{a_1\ldots a_{2j}}, \quad
 w^j_{a_1\ldots a_{2j}}\in {\mathcal E}_{(a_1\ldots a_{2j})}
\end{align*}
(with $w^j_{a_1\ldots a_{2j}}=0$ for almost all $j\in\mN_0$) and the symbol $v^a$ fulfilling
the differential system
\begin{align} \label{diffsystemsd}
\nabla_{(a}\nabla_bv_{c)}=0\quad\text{and}\quad\nabla^{[a}v^{b]_0}=0,
\end{align}
where the subscript $0$ indicates the ``trace-free'' part with respect to $\om.$ 

Prolonging this system, we obtain the bundle
$\ol{\mathcal{T}} := \cE_a \oplus \cE_{(ab)} \oplus \cE$ with the connection
\begin{align*} 
\ol{\nabla}_c\left( \begin{array}{ccc}
& v_a &  \\
w_{(ab)} & & \varphi     
\end{array} \right)=\left( \begin{array}{c}
\nabla_cv_a  - w_{ca} - \varphi\omega_{ca}   \\
\nabla_c w_{ab} \quad | \quad \nabla_c\varphi    
\end{array} \right)
\end{align*}
for $(v_a, w_{ab},\varphi) \in \cE_a \oplus \cE_{(ab)} \oplus \cE$. In particular, 
the solution space of the system \eqref{diffsystemsd}, i.e., the Lie algebra 
of first order symmetries of $\Dir_s$,  is isomorphic to 
the space of covariantly constant sections of $\ol{\mathcal{T}}$.
This ismorphism is given by the projection to the top slot in one direction and by 
the differential splitting 
\begin{align*} 
v_a\mapsto\left( 
\begin{array}{c}
v_a \\ \nabla_{(a}v_{b)} \quad \mid \quad \frac{1}{n}\omega^{kl}\nabla_kv_l     
\end{array} \right) .
\end{align*}
in the opposite direction. From this one can easily see that the solution space is a Lie algebra
given by the semidirect product of $\mathfrak{sp}(n) \oplus \mathbb{R}$ with its representation 
on $\mathbb{R}^{2n}$.

So far we discussed higher symmetries for projectively flat manifolds in the real dimension 2 and indicated 
an analogous problem for flat affine symplectic manifolds in higher dimensions.
The curved setting is far more complicated. However, we expect that at least the case of second order symmetries
is this task manageable in the sense that one can find symmetry operators explicitly 
on the assumption of certain curvature conditions.

We shall analyze these questions in more detail elsewhere.

\subsection*{Acknowledgments} 
P. Somberg and J. \v{S}ilhan acknowledge the financial support from the grant GA CR P201/12/G028.


\vspace{0.5cm} 

{Petr Somberg: Mathematical Institute of Charles
  University, \\ Sokolovsk\'a 83, Prague, Czech Republic,}
	{somberg@karlin.mff.cuni.cz}

\vspace{0.2cm}

{Josef \v{S}ilhan: Inst. of Math. and Stat. of Masaryk
  University, \\ Building 08, Kotl\'a\v{r}sk\'a 2, 611 37,
  Brno, Czech Republic}, {silhan@math.muni.cz} 
 
  \end{document}